\documentclass[10pt]{amsart}
\usepackage{color}
\usepackage{amsmath}
\usepackage{graphicx}
\usepackage{url}
\usepackage{amsfonts}
\usepackage{amssymb}
\providecommand{\U}[1]{\protect\rule{.1in}{.1in}}

\newcommand{\ulambda}{{\boldsymbol{\lambda}}}

\newtheorem{theorem}{Theorem}[section]

\newtheorem{corollary}[theorem]{Corollary}
\newtheorem{proposition}[theorem]{Proposition}

\theoremstyle{definition}
\newtheorem{example}[theorem]{Example}
\newtheorem{definition}[theorem]{Definition}
\theoremstyle{remark}
\newtheorem{remark}[theorem]{Remark}

\begin{document}
\title[Factorization of the canonical bases for higher level Fock spaces]{Factorization of the canonical bases for higher level Fock spaces}
\date{}
\author{Susumu Ariki, Nicolas Jacon and C\'{e}dric Lecouvey}
\maketitle
\begin{abstract}The level $l$ Fock space admits canonical bases $\mathcal{G}_{e}$ and
$\mathcal{G}_{\infty}$. They correspond to ${\mathcal{U}_{v}(\widehat
{\mathfrak
{sl}_{e}})}$ and $\mathcal{U}_{v}({\mathfrak{sl}}_{\infty})$-module
structures. We establish that the transition matrices relating these two bases
are unitriangular with coefficients in $\mathbb{N}[v]$. Restriction to the
highest weight modules generated by the empty $l$-partition then gives a
natural quantization of a theorem by Geck and Rouquier on the factorization of
decomposition matrices which are associated to Ariki-Koike algebras.
\end{abstract}

\section{Introduction}

In the classification of finite complex reflection groups by Shephard and Todd
\cite{ST}, there is a single infinite family of groups $G(lp,p,n)$
parametrized by the triples $(l,p,n)\in\mathbb{N}^{3}$ and $34$ other
``exceptional'' groups. If $p=1$, the group $G(l,1,n)$ is the wreath
product of the cyclic group of order $l$ with the symmetric group $S_{n}$. It
generalizes both the Weyl group of type $A_{n-1}$ (corresponding to the case
$l=1$) and the Weyl group of type $B_{n}$ ($l=2$). To $G(l,1,n)$ we may associate 
its Hecke algebra over the ring $A:=\mathbb{C}[q^{\pm1},Q_{1}^{\pm1}%
,\ldots,Q_{l}^{\pm1}]$, where $(q,Q_{1},\ldots,Q_{l})$ is an $l+1$-tuple of
indeterminates. This algebra can be seen as a deformation of the group algebra
of $G(l,1,n)$ and has applications to the modular representation theory of
finite reductive groups (see for example the survey \cite{Mathas}). As an
$A$-algebra, it has the set of generators $\{T_{0},\ldots,T_{n-1}\}$ such that the 
defining relations are  
\[
\prod_{i=1}^{l}(T_{0}-Q_{i})=0,\ (T_{i}-q)(T_{i}+1)=0,\ i=1,\ldots,n-1
\]
and the braid relations of type $B_{n}.$ We denote this algebra by $\mathcal{H}_{A}$. 
If we extend the scalars of $\mathcal{H}_{A}$ to $K=\mathbb{C}(q,Q_{1},\ldots,Q_{l})$, 
the field of fractions of $A$, we obtain the algebra $\mathcal{H}_{K}:=K\otimes_{A}\mathcal{H}_{A}$ whose representation theory is well understood. 
For example, we know how to classify the irreducible
representations, what are their dimensions etc (see \cite{AK}, \cite{HR}). 
The theory is far more difficult in the modular case.  
Let $\theta:A\rightarrow\mathbb{C}$ be a ring
homomorphism and let $\mathcal{H}_{\mathbb{C}}:=\mathbb{C}\otimes
_{A}\mathcal{H}_{A}$ be the associated Hecke algebra. Due to results of Dipper
and Mathas \cite{DM}, one can reduce various important problems to 
the case when $\theta(q)=\eta_{e}:=\text{exp}(\frac{2i\pi}{e})$ is a 
$e$-th root of unity, for $e\in\mathbb{Z}_{\geq2}$, and 
$\theta(Q_{j})=\eta_{e}^{s_{j}}$, for $j=1,\ldots, l$, where 
$(s_{1},\ldots,s_{l})\in\mathbb{Z}^{l}$. An important object of study in 
the modular case is the decomposition map. As $\mathcal{H}_{A}$ is a
cellular algebra \cite{GL}, the decomposition map may be defined as follows. Let
$V_{K}\in\text{Irr}(\mathcal{H}_{K})$. Then there exists a specific $\mathcal{H}_{A}
$-module $V_{A}$, which is called a \emph{cell module}, such that 
$V_{K}=K\otimes_{A}V_{A}$. We can then associate to
$V_{K}$ the $\mathcal{H}_{\mathbb{C}}$-module $V_{\mathbb{C}}=\mathbb{C}\otimes_{A}V_{A}$. 
This gives a well-defined map between Grothendieck groups $R_{0}(\mathcal{H}_{K})$ of 
finitely generated $\mathcal{H}_{K}$-modules and $R_{0}(\mathcal{H}_{\mathbb{C}})$ of 
finitely generated $\mathcal{H}_{\mathbb{C}}$-modules. We denote the decomposition map by 
\[
d_{\theta}: R_{0}(\mathcal{H}_{K}) \rightarrow  R_{0}(\mathcal{H}_{\mathbb{C}}).
\]
We denote the associated decomposition matrix by $D_{e}$. 
It is known that we may choose $V_{A}$ 
more general than the cell module and the decomposition map is still well defined \cite{Geckdecompo}. 

There exist algorithms to compute the map $d_{\theta},$ but it remains
difficult to describe it in general. One useful tool here is 
a result by Geck and Rouquier \cite{GRfac}, which 
gives information on the matrix $D_{e}$ by factorizing the decomposition map. 
Let $\theta^{q}:A\rightarrow\mathbb{C}(q)$ be the specialization map 
defined by $\theta^{q}(Q_{i})=q^{s_{i}}$, for $i=1,\ldots,l$. Denote
by $\mathcal{H}_{\mathbb{C}(q)}:=\mathbb{C}(q)\otimes_{A}\mathcal{H}_{A}$ the
associated Hecke algebra. As above, we have the decomposition map
\[
d_{\theta^{q}}:R_{0}(\mathcal{H}_{K})\rightarrow R_{0}(\mathcal{H}
_{\mathbb{C}(q)})
\]
and the associated decomposition matrix $D_{\infty}$. Then 
\cite[Prop. 2.12]{GRfac} implies the following. 

\begin{theorem}
[Geck-Rouquier]\label{factoh} There exists a unique ${\mathbb Z}$-linear map
\[
d_{\theta^{q}}^{\theta}\colon R_{0}(\mathcal{H}_{\mathbb{C}(q)})\rightarrow
R_{0}(\mathcal{H}_{\mathbb{C}})
\]
such that the following diagram commutes:

\[
\begin{picture}(190,59) \put(20,45){$R_0(\mathcal{H}_{K})$} \put
(65,48){\vector(1,0){80}} \put(100,54){$d_{\theta}$} \put
(155,45){$R_0(\mathcal{H}_{\mathbb{C}})$} \put(48,37){\vector(3,-2){35}}
\put(50,17){$d_{\theta^q}$} \put(128,13){\vector(3,2){35}} \put
(152,17){$d_{\theta^q}^{\theta}$} \put(83,05){$R_0(\mathcal{H}_{\mathbb{C}
(q)})$} \end{picture}
\]
\end{theorem}

Thus, we have the factorization $D_{e}=D_{\infty}.D_{\infty}^{e}$ of 
the decomposition matrices, where $D_{\infty}^{e}$ is the 
decomposition matrix for $d_{\theta^{q}}^{\theta}$. We shall call $D_{\infty}^{e}$ 
the \emph{relative decomposition matrix}. 
This result shows that a part of the representation theory of
$\mathcal{H}_{\mathbb{C}}$ does not depend on $e$ but only on the
representation theory of $\mathcal{H}_{\mathbb{C}(q)}$, which is ``easier'' to
understand (for example, there are closed formulae for the entries of
$D_{\infty}$ when $l=2$ \cite{LM}). An example of its application 
is that one may give explicit relationship among various classifications of simple
modules arising from the theory of canonical basic sets in type $B_{n}$
\cite{Jcons}. 

In view of the Fock space theory, which is now standard in the study of Hecke algebras, 
Theorem \ref{factoh} naturally leads to several questions. 
As noted above, there is an algorithm for computing the decomposition matrices
of $\mathcal{H}_{\mathbb{C}}$ and $\mathcal{H}_{\mathbb{C}(q)}$. This algorithm 
relies on the first author's proof (see \cite{arikilivre}) of the Lascoux-Leclerc-Thibon conjecture \cite{LLT}. 
His theorem asserts that
$D_{e}$ (resp. $D_{\infty}$) is equal to the evaluation at $v=1$ of the matrix
$D_{e}(v)$ (resp. $D_{\infty}(v)$) which is obtained by expanding the canonical basis in a
highest weight ${\mathcal{U}_{v}(\widehat{\mathfrak{sl}_{e}})}$-module (resp.
$\mathcal{U}_{v}(\mathfrak{sl}_{\infty})$-module) into linear combination of 
the standard basis of a Fock space. Thus, Theorem \ref{factoh} 
implies the existence of a matrix $D_{\infty}^{e}$ such that $D_{e}
(1)=D_{\infty}(1).D_{\infty}^{e}$. The entries of $D_{e}(v)$ and $D_{\infty}(v)$ 
are known to be in $\mathbb{N}[v]$, i.e. 
polynomials with nonnegative integer coefficients. Hence it is 
natural to ask:
\begin{itemize}
\item[(Q1)]
Does the matrix $D_{\infty}^{e}$ have a natural quantization ? 
Namely, is there a matrix $D_{\infty}^{e}(v)$ with entries in 
$\mathbb{N}[v] $ such that
$$D_{e}(v)=D_{\infty}(v).D_{\infty}^{e}(v)\;?$$
\item[(Q2)]
If $D_{\infty}^{e}(v)$ is known to exist, 
find a practical algorithm to compute $D_{\infty}^{e}(v)$.
\end{itemize}
In other words, we ask if the matrix
of the canonical basis for ${\mathcal{U}_{v}(\widehat{\mathfrak{sl}_{e}})}
$-modules factorizes through the matrix of the canonical basis for
$\mathcal{U}_{v}(\mathfrak{sl}_{\infty})$-modules. 

Highest weight ${\mathcal{U}_{v}(\widehat{\mathfrak{sl}_{e}})}$-modules and
$\mathcal{U}_{v}(\mathfrak{sl}_{\infty})$-modules are realized as
irreducible components of Fock spaces of higher level. By Uglov's results
\cite{U}, these Fock spaces also admit canonical bases. So the above questions 
also make sense for the matrices $\Delta_{e}(v)$ and 
$\Delta_{\infty}(v)$ which are associated to the canonical bases of the whole Fock
space. Thus, instead of (Q1), we ask whether 
there exists a matrix $\Delta_{\infty}^{e}(v)$ with entries in $\mathbb{N}[v]$ 
such that 
$$\Delta_{e}(v)=\Delta_{\infty}(v).\Delta_{\infty}^{e}(v).$$

The matrix $\Delta_{\infty}^{e}(v)$ is expected to have several interpretations. 
Observe that recent conjectures and results \cite{BJ}, \cite{BK},
\cite{BK2} show that $D_{e}(v)$ and $D_{\infty}(v)$ should be interpreted as  graded decomposition matrices 
of Hecke algebras. 
 $D_{\infty}^{e}(v)$ might also be interpreted as a graded analogue of $D_{\infty}^{e}$ 
in this setting. 
According to conjectures of Yvonne \cite{Y} and Rouquier \cite[\S 6.4]{Rou}, 
$\Delta_{e}(1)$ and $\Delta_{\infty}(1)$ are expected to be decomposition 
matrices of a generalized $\eta_e$ and $q$-Schur algebras, respectively. 
Thus, $\Delta_{\infty}^{e}(v)$ 
might have a similar meaning  as $D_{\infty}^{e}(v)$ as well. 

In another direction, we interpret the factorization 
$D_{e}=D_{\infty}.D_{\infty}^{e}$ in the context of 
parabolic BGG categories in the last section. 
This second interpretation should also have graded version, which is 
independent of the first (note that Hecke algebras are not 
positively graded.)

In this paper, we answer positively to the questions (Q1) and (Q2) for $\Delta_{\infty}^{e}(v)$. 
We first show the existence of the matrices $D_{\infty}^{e}(v)$ and 
$\Delta_{\infty}^{e}(v)$ with entries in $\mathbb{Z}[v]$. 
In fact $D_{\infty}^{e}(v)$ is a submatrix of $\Delta_{\infty}^{e}(v)$ 
and we provide an efficient algorithm for computing it 
(and thus an algorithm for computing $D_{\infty}^{e}$). Then, 
we prove that the entries of $\Delta_{\infty}^{e}(v)$ are in
$\mathbb{N}[v]$. More precisely, we show that they can be expressed as sum of
products of structure constants of the affine Hecke algebras of 
type $A$ with respect to the Kazhdan-Lusztig basis and its generalization by
Grojnowski-Haiman \cite{GH}.

Let us briefly summarize the main ingredients of our proofs. The Fock space
theory developed in \cite{jim} and the notion of canonical bases for these
Fock spaces introduced in \cite{U} make apparent strong connections between the
representation theories of ${\mathcal{U}_{v}(\widehat{\mathfrak{sl}_{e}})}$
and $\mathcal{U}_{v}(\mathfrak{sl}_{\infty})$. They permit us to prove the
existence of a matrix $\Delta_{\infty}^{e}(v)$ with entries in $\mathbb{Z}[v]$
such that $\Delta_{e}(v)=\Delta_{\infty}(v).\Delta_{\infty}^{e}(v).$ This
factorization can be regarded as an analogue, at the level of canonical bases,
of the compatibility of the crystal graph structures established in
\cite{JL}. It is achieved by introducing a new partial 
order on the set of $l$-partitions, which does not depend on $e$. This order  
differs from that used in \cite{U} and 
has the property that $\Delta_{e}(v)$ and $\Delta_{\infty}(v)$ 
are simultaneously unitriangular. 
The compatibility between the ${\mathcal{U}_{v}
(\widehat{\mathfrak{sl}_{e}})}$ and $\mathcal{U}_{v}(\mathfrak{sl}_{\infty}
)$-module structures on the Fock space then implies the factorization
$\Delta_{e}(v)=\Delta_{\infty}(v).\Delta_{\infty}^{e}(v)$. 
To show the positivity, recall that 
the coefficients of the matrices $\Delta_{\infty}(v)$ and $\Delta_{e}(v)$
are expressed by parabolic Kazhdan-Lusztig polynomials of the affine Hecke
algebras of type $A$ \cite{U}. We see in a simpler manner than \cite{U} how the 
parabolic Kazhdan-Lusztig polynomials are related to the entries of 
$\Delta_{\infty}(v)$ and $\Delta_{e}(v)$, for a fixed pair of $l$-partitions. 
The positivity result then follows from this and the positivity of the structure 
constants of the affine Hecke algebra.

\section{Background on Fock spaces and canonical bases}

We refer to \cite{kashi} and to \cite{arikilivre} for a detailed review on the 
canonical and crystal basis theory. \cite[\S7]{gecklivre} also gives a nice
survey on modular representation theory of Hecke algebras. Let $v$ be an
indeterminate, $e>1$ an integer, and ${\mathcal{U}_{v}(\widehat
{\mathfrak{sl}_{e}})}$ the quantum group of type $A_{e-1}^{(1)}$.
It is an associative $\mathbb{Q}(v)$-algebra with Chevalley generators $e_{i}
,f_{i},t_{i},t_{i}^{-1},$ for $i\in\mathbb{Z}/e\mathbb{Z},$ and $\partial$. We
refer to \cite[\S2.1]{U} for the precise definition. The bar-involution
$\overline{
\begin{tabular}
[c]{l}
\ \
\end{tabular}
\ }$ is the ring automorphism of ${\mathcal{U}_{v}(\widehat{\mathfrak{sl}_{e}
})}$ such that $\overline{v}=v^{-1},\overline{\partial}=\partial$ and, 
\[
\overline{e_{i}}=e_{i},\;\; \overline{f_{i}}=f_{i}\text{ and }
\overline{t_{i}}=t_{i}^{-1}, \;\;\text{for $i\in\mathbb{Z}/e\mathbb{Z}$}.
\]
We denote by ${\mathcal{U}_{v}^{\prime}(\widehat{\mathfrak{sl}_{e}})}$ the
subalgebra generated by $\{e_{i},f_{i},t_{i},t_{i}^{-1} \mid i\in\mathbb{Z}
/e\mathbb{Z}\}$. By slight abuse of notation, we identify the elements of
$\mathbb{Z}/e\mathbb{Z}$ with their corresponding labels in $\{0,\ldots,e-1\}
$ when there is no risk of confusion. Write $\{\Lambda_{0},\ldots
,\Lambda_{e-1}\}$ for the set of fundamental weights of ${\mathcal{U}
_{v}(\widehat{\mathfrak{sl}_{e}})}$, and $\delta$ for the null root. 
Let $l\in\mathbb{Z}_{\geq1}$ and consider
${\mathbf{s}}=(s_{1},\ldots,s_{l})\in\mathbb{Z}^{l}$, which we call a 
\emph{multicharge}. We set 
$$\mathfrak{s}=(s_{1}(\text{mod }e),\ldots,s_{l}(\text{mod }e))\in(\mathbb{Z}
/e\mathbb{Z})^{l}$$ and $\Lambda_{\mathfrak{s}}:=\Lambda_{s_{1}(\text{mod }
e)}+\ldots+\Lambda_{s_{l}(\text{mod }e)}$.

Similarly, let ${\mathcal{U}_{v}(\mathfrak{sl}_{\infty})}$ be the quantum
group of type $A_{\infty}.$ It is an associative $\mathbb{Q}(v)$-algebra
with Chevalley generators $E_{j},F_{j},T_{j},T_{j}^{-1},$ for $j\in\mathbb{Z}$. We 
use the same symbol $\overline{
\begin{tabular}
[c]{l}
\ \
\end{tabular}
\ }$ to denote its bar-involution, which is the ring automorphism of ${\mathcal{U}_{v}(\mathfrak{sl}_{\infty})}$ such
that $\overline{v}=v^{-1}$ and,
\[
\overline{E_{j}}=E_{j},\;\;\overline{F_{j}}=F_{j}\text{ and }\overline{T_{j}
}=T_{j}^{-1},\;\;\text{for $j\in\mathbb{Z}$}.
\]
Write $\{\omega_{j},j\in\mathbb{Z\}}$ for its set of fundamental weights. To
${\mathbf{s}}=(s_{1},\ldots,s_{l})\in\mathbb{Z}^{l}$, we associate the
dominant weight $\Lambda_{{\mathbf{s}}}:=\omega_{s_{1}}+\cdots+\omega
_{s_{l}}$.

\subsection{Fock spaces}

Let $\Pi_{l,n}$ be the set of $l$-partitions with rank $n,$ that is, the set of
sequences ${\boldsymbol{\lambda}=}(\lambda^{(1)},\ldots,\lambda^{(l)})$ of
partitions such that $\left|  {\boldsymbol{\lambda}}\right|  =\left|
\lambda^{(1)}\right|  +\cdots+\left|  \lambda^{(l)}\right|  =n$. Set
$\Pi_{l}=\cup_{n\geq0}\Pi_{l,n}$. We also write $\Pi=\cup_{n\geq0}\Pi_{1,n}$ 
for short.  
The \emph{Fock space} $\mathcal{F}$ of level $l$ 
is a $\mathbb{Q}(v)$-vector space which has the set of all $l$-partitions as 
the given basis, so that we write
\[
\mathcal{F}=\bigoplus_{{\boldsymbol{\lambda}}\in\Pi_{l}
}\mathbb{Q}(v){\boldsymbol{\lambda}}.
\]
The Fock space $\mathcal{F}$ may be endowed with a structure of
${\mathcal{U}_{v}(\widehat{\mathfrak{sl}_{e}})}$ and ${\mathcal{U}
_{v}(\mathfrak{sl}_{\infty})}$-modules. Let ${\boldsymbol{\lambda}}$ be an
$l$-partition (identified with its Young diagram). Then, the \emph{nodes} of
${\boldsymbol{\lambda}}$ are the triples $\gamma=(a,b,c)$ where
$c\in\{1,\ldots,l\}$ and $a,b$ are the row and column indices
of the node $\gamma$ in $\lambda^{(c)}$, respectively. 
The \emph{content} of $\gamma$ is the
integer $c\left(  \gamma\right)  =b-a+s_{c}$ and the \emph{residue} $\mathrm{res(}
\gamma)$ of $\gamma$ is the element of $\mathbb{Z}/e\mathbb{Z}$ such that
\begin{equation}
\mathrm{res}(\gamma)\equiv c(\gamma)(\text{mod }e). \label{res}
\end{equation}

For $i\in\mathbb{Z}/e\mathbb{Z}$, we say that $\gamma$ is an $i$-node
of ${\boldsymbol{\lambda}}$ when $\mathrm{res}(\gamma)\equiv i(\text{mod }e).$ 
Similarly for $j\in\mathbb{Z}$, we say that 
$\gamma$ is a $j$-node of ${\boldsymbol{\lambda}}$
when $c(\gamma)=j.$ We say that a node $\gamma$ is \emph{removable} when $\gamma
=(a,b,c)\in{\boldsymbol{\lambda}}$ and ${\boldsymbol{\lambda}}\backslash
\{\gamma\}$ is an $l$-partition, and \emph{addable} when
$\gamma=(a,b,c)\notin{\boldsymbol{\lambda}}$ and ${\boldsymbol{\lambda}}
\cup\{\gamma\}$ is an $l$-partition.

Let $i\in\mathbb{Z}/e\mathbb{Z}$. In the sequel, we follow the convention of
\cite{U}. We define a total order 
on the set of $i$-nodes of ${\boldsymbol{\lambda}}$. Consider two
nodes $\gamma_{1}=(a_{1},b_{1},c_{1})$ and $\gamma_{2}=(a_{2},b_{2},c_{2})$ in ${\boldsymbol{\lambda}}$. We define the order by
\[
\gamma_{1}\prec_{{\mathbf{s}}}\gamma_{2}\Longleftrightarrow\left\{
\begin{array}
[c]{l}
c(\gamma_{1})<c(\gamma_{2})\text{ or}\\
c(\gamma_{1})=c(\gamma_{2})\text{ and }c_{1}<c_{2}.
\end{array}
\right.
\]
Let ${\boldsymbol{\lambda}}$ and ${\boldsymbol{\mu}}$ be two $l$-partitions of
rank $n$ and $n+1$ such that $[{\boldsymbol{\mu}}]=[{\boldsymbol{\lambda}
}]\cup{\{\gamma\}}$ where $\gamma$ is an $i$-node. Define
\begin{align}
{N}_{i}^{\succ}{({\boldsymbol{\lambda}},{\boldsymbol{\mu}})}=  &
\sharp\{\text{addable }\ i\text{-nodes }\gamma^{\prime}\text{ of }{\boldsymbol{\lambda
}}\ \text{ such that }\gamma^{\prime}\succ_{{\mathbf{s}}}\gamma\}\label{Ni}\\
&  -\sharp\{\text{removable }\ i\text{-nodes }\gamma^{\prime}\text{ of
}{\boldsymbol{\mu}}\ \text{ such that }\gamma^{\prime}\succ_{{\mathbf{s}}
}\gamma\},\nonumber\\
{N}_{i}^{\prec}{({\boldsymbol{\lambda}},{\boldsymbol{\mu}})}=  &
\sharp\{\text{addable }i\text{-nodes }\gamma^{\prime}\text{ of }
{\boldsymbol{\lambda}}\ \text{ such that }\gamma^{\prime}\prec_{{\mathbf{s}}
}\gamma\}\nonumber\\
&  -\sharp\{\text{removable }i\text{-nodes }\gamma^{\prime}\text{ of
}{\boldsymbol{\mu}}\ \text{ such that }\gamma^{\prime}\prec_{{\mathbf{s}}
}\gamma\},\\
{N}_{i}{({\boldsymbol{\lambda}})}=  &  \sharp\{\text{addable }i\text{-nodes of
}{\boldsymbol{\lambda}}\}\nonumber\\
&  -\sharp\{\text{removable }i\text{-nodes of }{\boldsymbol{\lambda}}\}\text{
}\nonumber\\
\text{ and }{M}_{0}{({\boldsymbol{\lambda}})}=  &  \sharp\{0\text{-nodes of
}{\boldsymbol{\lambda}}\}.\nonumber
\end{align}

\begin{theorem}
\cite{jim}\label{jmm} Let $\mathbf{s}\in\mathbb{Z}^{l}$. The Fock space
$\mathcal{F}$ has a structure of an integrable ${\mathcal{U}_{v}
(\widehat{\mathfrak{sl}_{e}})}$-module $\mathcal{F}_{e}^{\mathbf{s}}$ defined
by
\begin{gather*}
e_{i}{\boldsymbol{\lambda}}=\sum_{\text{res}([{\boldsymbol{\lambda}}]/[{\boldsymbol{\mu}}])=i}{v^{-{N}_{i}^{\prec}{({\boldsymbol{\mu}
},{\boldsymbol{\lambda}})}}{\boldsymbol{\mu}},\quad{f_{i}{\boldsymbol
{\lambda}}=\sum_{\text{res}([{\boldsymbol{\mu}}]/[{\boldsymbol{\lambda}}
])=i}{v^{{N}_{i}^{\succ}{({\boldsymbol{\lambda}},{\boldsymbol{\mu}})}}}}
}{\boldsymbol{\mu}},\\[5pt]
\qquad t_{i}{\boldsymbol{\lambda}}=v^{{N}_{i}{({\boldsymbol{\lambda}})}}
{\boldsymbol{\lambda}}\quad\text{and}\quad\partial{\boldsymbol{\lambda}}=-(\Delta
+M_{0}({\boldsymbol{\lambda}})){\boldsymbol{\lambda}},
\end{gather*}
for $i\in\mathbb{Z}/e\mathbb{Z}$, 
where $\Delta$ is the rational number defined in \cite[Thm 2.1]{jim}. The module
structure on $\mathcal{F}_{e}^{\mathbf{s}}$ depends on $\mathbf{s}$ and $e.$
\end{theorem}

\noindent We may consider $\mathcal{F}$ as a 
${\ \mathcal{U}_{v}^{\prime}(\widehat{\mathfrak{sl}_{e}})}
$-module by restriction. We denote it by 
the same $\mathcal{F}_{e}^{\mathbf{s}}$ by abuse of notation.

Let $j\in\mathbb{Z}$. For $l$-partitions ${\boldsymbol{\lambda}}$ and ${\boldsymbol
{\mu}}$ of rank $n$ and $n+1$ such that 
$[{\boldsymbol{\mu}}]=[{\boldsymbol{\lambda}}]\cup{\{\gamma\}}$ 
where $\gamma$ is a $j$-node, we define 
${N}_{j}^{\succ}{({\boldsymbol{\lambda}},{\boldsymbol{\mu}}),}$
${N}_{j}^{\prec}{({\boldsymbol{\lambda}},{\boldsymbol{\mu}})}$ and ${N}
_{j}{({\boldsymbol{\lambda}})}$ as in (\ref{Ni}) except that we consider $j$-nodes, 
for $e=\infty$, instead of $i$-nodes, for $e$ finite. 

\begin{theorem}
\cite{jim} \label{TH_Jim}Let $\mathbf{s}\in\mathbb{Z}^{l}$. The Fock space
$\mathcal{F}$ has a structure of an integrable ${\mathcal{U}_{v}(\mathfrak
{sl}_{\infty})}$-module $\mathcal{F}_{\infty}^{\mathbf{s}}$ defined by
\begin{gather*}
E_{j}{\boldsymbol{\lambda}}=\sum_{c([{\boldsymbol{\lambda}}]/[{\boldsymbol
{\mu}}])=j}{v^{-{N}_{j}^{\prec}{({\boldsymbol{\mu}},{\boldsymbol{\lambda}})}
}{\boldsymbol{\mu}}, \quad 
F{_{j}{\boldsymbol{\lambda}}=\sum_{c([{\boldsymbol
{\mu}}]/[{\boldsymbol{\lambda}}])=j}{v^{{N}_{j}^{\succ}{({\boldsymbol{\lambda
}},{\boldsymbol{\mu}})}}}}}{\boldsymbol{\mu}},\\[5pt]
T_{j}{\boldsymbol{\lambda}}=v^{{N}_{j}{({\boldsymbol{\lambda}})}}
{\boldsymbol{\lambda}},
\end{gather*}
for $j\in\mathbb{Z}$. 
The module structure on $\mathcal{F}_{e}^{\mathbf{s}}$ depends on $\mathbf{s}.$
\end{theorem}

The following result is implicit in \cite[Prop 3.5]{jim}.

\begin{proposition}
\label{Prop_compa_action}The ${\mathcal{U}_{v}^{\prime}(\widehat{\mathfrak
{sl}_{e}})}$ and ${\mathcal{U}_{v}(\mathfrak{sl}_{\infty})}$-module
structures $\mathcal{F}_{e}^{\mathbf{s}}$ and $\mathcal{F}_{\infty
}^{\mathbf{s}}$ are compatible in the sense that we may write 
the action of $e_{i}, f_{i}$ and $t_{i}$, for $i\in\mathbb{Z}/e\mathbb{Z}$, as follows:
\begin{align*}
e_{i} &  =\sum_{j\in\mathbb{Z},j\equiv i(\text{mod }e)}\left(  \prod_{r\geq
1}T_{j-re}^{-1}\right)  E_{j},\\
f_{i} &  =\sum_{j\in\mathbb{Z},j\equiv i(\text{mod }e)}\left(  \prod_{r\geq
1}T_{j+re}\right)  F_{j},\\
t_{i} &  =\prod_{j\in\mathbb{Z},j\equiv i(\text{mod }e)}T_{j}.
\end{align*}
\end{proposition}

\begin{remark}
The infinite sums and products in the proposition reduce in fact to
finite ones since the number of nodes in ${\boldsymbol{\lambda}}$ is finite.
\end{remark}

The empty multipartition $\mathbf{\emptyset}$ is a highest weight vector in
$\mathcal{F}_{e}^{\mathbf{s}}$ and $\mathcal{F}_{\infty}^{\mathbf{s}}$ of
weight $\Lambda_{\mathfrak{s}}$ and $\Lambda_{\mathbf{s}}$, respectively. We
then define $V_{e}({\mathbf{s}})$ and $V_{\infty}({\mathbf{s}})$ as the 
highest weight modules
$\mathcal{U}^{\prime}_{v}(\widehat{\mathfrak{sl}_{e}}).\mathbf{\emptyset} $
and ${\mathcal{U}_{v}(\mathfrak{sl}_{\infty}).}\mathbf{\emptyset}$,
respectively. Observe that the module structure on $V_{e}({\mathbf{s}})$ 
really depends on $\mathbf{s}$ and not only on its class $\mathfrak{s}$ modulo $e$. 
By the previous proposition, 
it follows that $V_{\infty}({\mathbf{s}})$ is endowed 
with the structure of a 
$\mathcal{U}^{\prime}_{v}(\widehat{\mathfrak{sl}_{e}})$-module and 
$V_{e}({\mathbf{s}})$ coincides with the $\mathcal{U}^{\prime}
_{v}(\widehat{\mathfrak{sl}_{e}})$-submodule of $V_{\infty
}({\mathbf{s}})$ generated by the highest weight vector $\mathbf{\emptyset}$.

\subsection{Uglov's canonical bases}

We now briefly recall Uglov's plus canonical basis of the Fock
spaces. Let $\mathbb{A}(v)$ be the ring of rational functions which have no pole at
$v=0$. Set
\begin{align*}
\mathcal{L}  &  :=\bigoplus_{n\geq0}\bigoplus_{{\boldsymbol{\lambda}}\in
\Pi_{l,n}}\mathbb{A}(v){\boldsymbol{\lambda}}\;\;\text{ and}\\
\mathcal{B}  &  :=\{{\boldsymbol{\lambda}}\text{ }(\text{mod }v\mathcal{L}
)\mid{\boldsymbol{\lambda}\in}\Pi_{l}\}.
\end{align*}

\begin{theorem}
\cite{FLOTW} The pair $(\mathcal{L},\mathcal{B})$ is a crystal basis for
$\mathcal{F}_{e}^{\mathbf{s}}$ and $\mathcal{F}_{\infty}^{\mathbf{s}}.$
\end{theorem}

Note that although the crystal lattice $\mathcal{L}$ and 
the basis $\mathcal{B}$ of $\mathcal{L}/v\mathcal{L}$ are the same for
$\mathcal{F}_{e}^{\mathbf{s}}$ and $\mathcal{F}_{\infty}^{\mathbf{s}}$, 
the induced crystal structures $\mathcal{B}_{e}$ and $\mathcal{B}
_{\infty}$ on $\mathcal{B}$ do not coincide. The crystal structure $\mathcal{B}
_{e}$ is obtained as follows. Let ${\boldsymbol{\lambda}}$ be an $l$-partition,  
and $i\in\mathbb{Z}/e\mathbb{Z}$. We consider the set of addable
and removable $i$-nodes of ${\boldsymbol{\lambda}}$. 
We read the nodes in the increasing order with respect to $\prec_{{\mathbf{s}}}$, 
and let $w_{i}$ be the resulting word of the nodes. 
If a removable $i$-node appears just before an addable $i$-node, we delete both 
and continue the same procedure as many times as possible. In the end, we reach a word 
$\widetilde{w}_{i}$ of nodes such that the first $p$ nodes are addable and the last $q$ nodes are removable, for some $p, q\in\mathbb{N}$. 
If $p>0,$ let $\gamma$ be
the rightmost addable $i$-node in $\widetilde{w}_{i}$. The node $\gamma$ is called the \emph{good} $i$-node of $\boldsymbol{\lambda}$. Then, the
crystal $\mathcal{B}_{e}$ may be read off from its crystal graph:

\begin{itemize}
\item  vertices: $l$-partitions whose nodes are colored with residues.

\item  edges: $\displaystyle{{\boldsymbol{\lambda}}}\overset{i}{\rightarrow
}{{\boldsymbol{\mu}}}$ if and only if ${\boldsymbol{\mu}}$ is obtained by
adding a good $i$-node to ${\boldsymbol{\lambda}}$.
\end{itemize}

We denote by $B_{e}({\mathbf{s}})$ the connected component of 
$\mathcal{B}_{e}$ which contains the highest weight vertex $\mathbf{\emptyset}$. 
We may identify $B_{e}({\mathbf{s}})$ with the crystal graph of 
$V_{e}({\mathbf{s}}).$ The crystal graph of $\mathcal{B}_{\infty}$ is obtained in a similar manner: 
we use $j$-nodes ($j\in\mathbb{Z}$), for $e=\infty$, instead of $i$-nodes, for
$e$ finite. We may also identify the crystal graph of $V_{\infty}({\mathbf{s}})$
with $B_{\infty}({\mathbf{s}})$, the connected component of $\mathcal{B}
_{\infty}$ which contains the highest weight vertex $\mathbf{\emptyset}$.

\medskip

Let $e\in\mathbb{Z}_{\geq2}\cup\{\infty\}.$ We denote 
$$
{\mathcal{U}_{v}(\mathfrak{g})}=
\begin{cases}
{\mathcal{U}_{v}^{\prime}(\widehat{\mathfrak{sl}_{e}})}\quad&(\text{if $e<\infty$})\\
{\mathcal{U}_{v}(\mathfrak{sl}_{\infty})}\quad&(\text{if $e=\infty$})
\end{cases}
$$
for short. We define a $\mathbb{Z}[v]$-lattice 
$\mathcal{L}_{\mathbb{Z}}$ of $\mathcal{L}$ by
\[
\mathcal{L}_{\mathbb{Z}}:=\bigoplus_{n\geq0}\bigoplus_{{\boldsymbol{\lambda}
}\in\Pi_{l,n}}\mathbb{Z}[v]{\boldsymbol{\lambda}.}
\]
In \cite{U}, Uglov introduced a bar-involution $\overline{
\begin{tabular}
[c]{l}
\ \
\end{tabular}
\ }$ on $\mathcal{F}_{e}^{\mathbf{s}}$, which is defined by 
\[
\overline{u.f}=\overline{u}.\overline{f},\text{ for }u\in{\mathcal{U}
_{v}(\mathfrak{g})}\text{ and }f\in\mathcal{F}_{e}^{\mathbf{s}},\;\;
\text{and $\overline{\mathbf{\emptyset}}=\mathbf{\emptyset}$.}
\]
Such a bar-involution is easier to define for the Fock space 
$\mathcal{F}_{\infty}^{\mathbf{s}}$ as is explained in \cite[\S 3.9]{BK}. In the two cases, this leads 
 to the following Theorem-definition.

\begin{theorem}
\cite{U}\label{TH_Bc_Fock} Let $\mathbf{s}\in\mathbb{Z}^{l}$ and
$e\in\mathbb{Z}_{\geq2}\cup\{\infty\}.$ There exists a unique basis 
$\mathcal{G}_{e}(\mathbf{s})=\{G_{e}({\boldsymbol{\lambda},\mathbf{s})}
\mid{\boldsymbol{\lambda}\in}\Pi_{l}\}$ of $\mathcal{F}_{e}^{\mathbf{s}}$ 
such that the basis elements are characterized by the following two conditions.

\begin{enumerate}
\item $\overline{G_{e}({\boldsymbol{\lambda},\mathbf{s})}}{=}G_{e}
({\boldsymbol{\lambda},\mathbf{s}),}$

\item $G_{e}({\boldsymbol{\lambda},\mathbf{s})\equiv\boldsymbol{\lambda}}$ $($mod
$v\mathcal{L}_{\mathbb{Z}}).$
\end{enumerate}
\end{theorem}

The basis $\mathcal{G}_{e}(\mathbf{s})$ is called the \emph{plus canonical basis} of
$\mathcal{F}_{e}^{\mathbf{s}}$. It strongly depends on $e\in\mathbb{Z}_{\geq
2}\cup\{\infty\}.$ The purpose of the next theorem is to identify the Kashiwara-Lusztig 
canonical basis of $V_{e}({\mathbf{s}})$ with a subset of $\mathcal{G}_{e}(\mathbf{s}).$

\begin{theorem}
\cite{U}\label{Th_U_Comp} Let $\mathbf{s}\in \mathbb{Z}^{l}$ and $e\in 
\mathbb{Z}_{\geq 2}\cup \{\infty \}.$ Define 
\begin{equation*}
\mathcal{G}_{e}^{\circ }(\mathbf{s})=\mathcal{G}_{e}(\mathbf{s})\cap V_{e}
(\mathbf{s}).
\end{equation*}
Then $\mathcal{G}_{e}^{\circ }(\mathbf{s})$ coincides with the canonical
basis of the irreducible highest weight ${\mathcal{U}_{v}(\mathfrak{g})}$
-module$\ V_{e}({\mathbf{s}})$. Moreover, $G_{e}({\boldsymbol{\lambda
},\mathbf{s})\in }\mathcal{G}_{e}^{\circ }(\mathbf{s})$ if and only if 
${\boldsymbol{\lambda}\in }B_{e}({\mathbf{s}})$.
\end{theorem}

\section{Compatibility of canonical bases}

In this section, we prove that each $G_{e}(\boldsymbol{\lambda},\mathbf{s})$
may be expanded into $\mathbb{Z}[v]$-linear combination of the canonical basis
$\mathcal{G}_{\infty}(\mathbf{s})$. A crucial observation for the proof is
that we may define a partial order on multipartitions which is independent of
$e$. Then, the transition matrix becomes unitriangular with respect to the
partial order.

\subsection{Some combinatorial preliminaries}

\label{subsec_bij}

A \emph{$1$-runner abacus} is a subset $A$ of $\mathbb{Z}$ such that $-k\in A
$ and $k\notin A$ for all large enough $k\in\mathbb{N}$. To visualize a
$1$-runner abacus, we view $\mathbb{Z}$ as a horizontal runner and place a
bead on the $k$-th position, for each $k\in A$. Thus, the runner is full of
beads on the far left and has no beads on the far right. For $l\geq1$, an
\emph{$l $-runner abacus} is an $l$-tuple of $1$-runner abaci. Let
$\mathcal{A}^{l}$ be the set of $l$-runner abaci. To each pair of an
$l$-partition ${\boldsymbol{\lambda}}=(\lambda^{(1)},\ldots,\lambda^{(l)})$
and a multicharge $\mathbf{s}=(s_{1},\ldots,s_{l} )\in\mathbb{Z}^{l},$ we
associate the $l$-runner abacus
\[
a(\boldsymbol{\lambda},\mathbf{s}):=\{{(\lambda_{i}^{(d)}+s_{d}+1-i,d) \mid
i\geq1,1\leq d\leq l\}},
\]
which is a subset of $\mathbb{Z}\times\lbrack1,l]$. One checks easily that the
map
\[
(\boldsymbol{\lambda},\mathbf{s})\in\Pi_{l}\times\mathbb{Z}^{l}\mapsto
a(\boldsymbol{\lambda},\mathbf{s})\in\mathcal{A}^{l}%
\]
is bijective. To describe the embedding of Fock spaces into the space of
semi-infinite wedge products and then cut semi-infinite wedge products to
finite wedge products, we need to introduce a bijective map $\tau_{l}%
:\Pi\times\mathbb{Z} \cong\mathcal{A}\rightarrow\mathcal{A}^{l}\cong\Pi
_{l}\times\mathbb{Z}^{l}.$

\begin{definition}
\label{definition_tau_l} Let $\tau_{l}:\mathbb{Z}\rightarrow\mathbb{Z}%
\times\lbrack1,l]$ be the bijective map defined by
\[
k\mapsto(\phi(k),d(k)),
\]
where $k=c(k)+e(d(k)-1)+elm(k)$ such that
\[
c(k)\in\lbrack1,e],\;\;d(k)\in\lbrack1,l]\;\text{ and }\;m(k)\in\mathbb{Z},
\]
and $\phi(k)=c(k)+em(k).$ Then, we define $\tau_{l}:\Pi\times\mathbb{Z}%
\cong\mathcal{A}\rightarrow\mathcal{A}^{l}\cong\Pi_{l}\times\mathbb{Z}^{l}$
by
\[
A\mapsto\tau_{l}(A)=\{(\phi(k),d(k))\mid k\in A\}\in\mathcal{A}^{l},
\]
for $A\in\mathcal{A}$.
\end{definition}

\begin{remark}
\ 

\begin{enumerate}
\item  If $(\boldsymbol{\lambda},\mathbf{s})=\tau_{l}(\lambda,s)$ then
$s=s_{1}+\cdots+s_{l}$.

\item  To read off the multicharge $\mathbf{s}=(s_{1},\ldots,s_{l})$ from the
$l$-runner abacus, we proceed as follows: if the left adjacent position of a
bead on a runner is vacant, we move the bead to the left to occupy the vacant
position, and we repeat this procedure as many times as possible. Then,
$s_{d}$ is the column number of the rightmost bead of the $d$-th runner.
\end{enumerate}
\end{remark}

\begin{example}
Let $e=2$ and $l=3$. Then the preimage of
\[
(\boldsymbol{\lambda},\mathbf{s})=(((1.1),(1.1),(1)),(0,0,-1))
\]
is $(\lambda,s)=((4.3.3.2.1),-1).$
\end{example}

\vspace*{5mm} \begin{figure}[ptbh]
\begin{center}
\includegraphics[height=4.5cm]{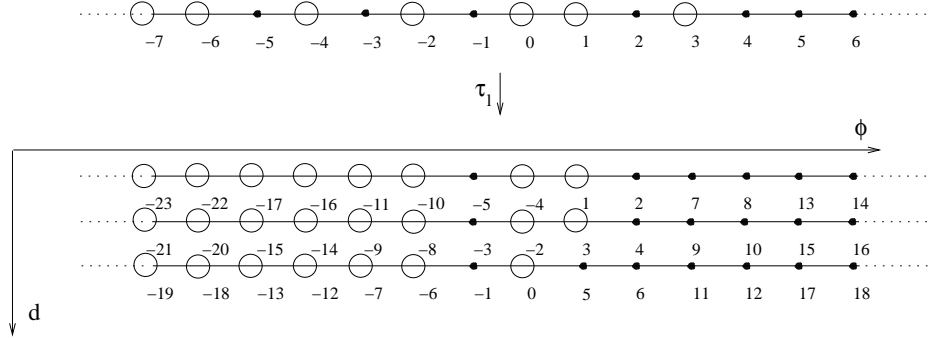}
\end{center}
\caption{Computation of the bijection $\tau_{l}$ using abaci.}%
\label{fig:1}%
\end{figure}

Now, $(\lambda,s)=\tau_{l}^{-1}(\boldsymbol{\lambda},\mathbf{s})$ has the
$1$-runner abacus
\[
a(\lambda,s)=\{(k_{i}:=\lambda_{i}+s+1-i) \mid i\geq1\},
\]
and the semi-infinite sequence $(k_{1},k_{2},\dots)$ defines a semi-infinite
wedge product.

We fix a sufficiently large $r$ such that
\begin{equation}
\lambda_{i}=0,\text{ for }i\geq r. \label{condkr}%
\end{equation}
Then $(\lambda,s)$ is determined by the finite sequence $\mathbf{k}%
:=(k_{1},\dots,k_{r}).$ For example, $((4.3.3.2.1),-1)$ is determined by
$\mathbf{k}=(3,1,0,-2,-4,-6,-7).$ We write $\mathbf{k}=\tau_{l}^{-1}%
(\boldsymbol{\lambda},\mathbf{s})$ by abuse of notation. Then they give the
wedge basis in the space of finite wedge products $\Lambda^{r}$, which will be
introduced in a different guise in \S\ref{Expression in KL}.

We read the beads $\tau_{l}(k_{1}),\dots,\tau_{l}(k_{r})$ on the $l$-runner
abacus $a(\boldsymbol{\lambda},\mathbf{s})$ from right to left, starting with
the $l$-th runner, and obtain a permutation $\mathrm{w}(\mathbf{k}%
)=(w_{1},\ldots,w_{r})$ of $\mathbf{k}.$ In our example, we have
$\mathrm{w}(\mathbf{k})=(0,-6,-7,3,-2,1,-4).$

\begin{definition}
Let $\tau_{l}(w_{i})=(\zeta_{i},b_{i})$, for $1\leq i\leq r$, that is,
$\zeta_{i}$ and $b_{i}$ are the column number and the row number of the bead
$\tau_{l}(w_{i})$ on the $l$-runner abacus $a(\boldsymbol{\lambda}%
,\mathbf{s})$, respectively. Then, we define
\[
\zeta(\boldsymbol{\lambda})=(\zeta_{1},\ldots,\zeta_{r})\;\text{ and
}\;b(\boldsymbol{\lambda})=(b_{1},\ldots,b_{r}).
\]
\end{definition}

\begin{example}
In our example, we have
\[
\zeta(\boldsymbol{\lambda})=(0,-2,-3,1,0,1,0)\;\text{ and }\;b(\boldsymbol
{\lambda})=(3,3,3,2,2,1,1).
\]
\end{example}

We will need the $\zeta(\boldsymbol{\lambda})$ and $b(\boldsymbol{\lambda})$
when we express $\Delta_{{\boldsymbol{\lambda}},{\boldsymbol{\mu}}}^{e}(v)$ in
Kazhdan-Lusztig polynomials. In this respect, the following remark is important.

\begin{remark}
\label{rq_inde} Suppose that we have fixed $\boldsymbol{\lambda}$ and
$\mathbf{s}$. Assume $e$ and $e^{\prime}$ are two positive integers. Then
$\mathbf{k}=\tau_{l}^{-1}(\boldsymbol{\lambda},\mathbf{s})$ does not coincide
in general for distinct $e$ and $e^{\prime}$. Nevertheless, one can choose $r$
such that $\zeta(\boldsymbol{\lambda})$ and $b(\boldsymbol{\lambda
})$ for $e$ \emph{coincide} with those for $e^{\prime}$. For this to hold, it suffices that 
the $r$ beads are the same for $e$ and $e^{\prime}$. Thus, it suffices to choose $r$ as in
(\ref{condkr}) such that $1-k_r$ is divisible by $e$ and $e^{\prime}$. 
If we divide the $l$-runner abacus into \emph{cells} with
height $l$ and width $e$ (resp. $e^{\prime}$) so that the initial cell
contains exactly the locations labelled by $1,2\ldots,el$ (resp.
$1,2\ldots,e^{\prime}l$), 
it says that the finite sequence ends at the upper-left corner of a far left cell for 
both $e$ and $e^{\prime}$. In our running example, if we want to make
$\zeta(\boldsymbol{\lambda})$ and $b(\boldsymbol{\lambda
})$ coincide for $e=2$ and $e^{\prime}=3,$ we read all the beads with labels
greater or equal to $-17$ in Figure \ref{fig:1}.
\end{remark}

Let $P=\mathbb{Z}^{r}$ and $W$ the affine symmetric group which is the
semidirect product of the symmetric group $S_{r}$ and the normal subgroup $P.
$ $W$ acts on $\beta=(\beta_{1},\ldots,\beta_{r})\in P$ on the right by
\begin{align}
\beta\cdot s_{i}  &  =(\beta_{1}\cdots,\beta_{i+1},\beta_{i},\ldots,\beta
_{r}),\;\text{for $1\leq i\leq r-1,$ and}\label{right_action}\\
\beta\cdot\mu &  =\beta+e\mu,\;\text{for $\mu\in P$.}\nonumber
\end{align}
Then
\[
A^{r}=\{a=(a_{1},\ldots,a_{r})\in P\mid1\leq a_{1}\leq\cdots\leq\cdots
a_{r}\leq e\}
\]
is a fundamental domain for the action. We denote the stabilizer of $a\in
A^{r}$ by $_{a}\!W$. It is clear that $_{a}\!W$ is a subgroup of $S_{r}$. Let
$w_{a}$ be the maximal element of $_{a}\!W.$ We denote by $^{a}W$ and $^{a}S_{r}
$ the set of minimal length coset representatives in $_{a}\!W\backslash W$ and
$_{a}\!W\backslash S_{r}$, respectively.

In a similar manner, $W$ acts on $\beta=(\beta_{1},\ldots,\beta_{r})\in P$ on
the left by
\begin{align*}
s_{i}\cdot\beta &  =(\beta_{1}\cdots,\beta_{i+1},\beta_{i},\ldots,\beta
_{r}),\; \text{for $1\leq i\leq r-1$, and}\\
\mu\cdot\beta &  =\beta+l\mu,\;\text{for $\mu\in P$.}%
\end{align*}
Then
\[
B^{r}=\{b=(b_{1},\ldots,b_{r})\in P\mid l\geq b_{1}\geq\cdots\geq\cdots
b_{r}\geq1\}
\]
is a fundamental domain for the action. We denote the stabilizer of $b\in
B^{r}$ by $W_{b}$, its maximal element by $w_{b},$ and the set of minimal
length coset representatives in $W/W_{b}$ and $S_{r}/W_{b}$ by $W^{b}$ and
$S_{r}^{b}$, respectively.

\medskip

Write $k=c(k)+e(d(k)-1)+elm(k)$ and $\phi(k)=c(k)+em(k)$, for $k\in\mathbb{Z}
$, as before, and define
\begin{align*}
c(\mathbf{k})  &  =(c(k_{1}),\ldots,c(k_{r})),\\
d(\mathbf{k})  &  =(d(k_{1}),\ldots,d(k_{r})),\\
m(\mathbf{k})  &  =(m(k_{1}),\ldots,m(k_{r})),\\
\phi(\mathbf{k})  &  =(\phi(k_{1}),\ldots,\phi(k_{r})),
\end{align*}
for $\mathbf{k}=\tau_{l}^{-1}(\boldsymbol{\lambda},\mathbf{s})\in
\mathbb{Z}^{r}.$ Then,

\begin{itemize}
\item  there exist $a(\mathbf{k})\in A^{r}$ and $u(\mathbf{k})\in
{^{a(\mathbf{k})}}S_{r}$ such that $c(\mathbf{k})=a(\mathbf{k})\cdot
u(\mathbf{k}).$

\item  there exist $b(\mathbf{k})\in B^{r}$ and $v(\mathbf{k})\in
S_{r}^{b(\mathbf{k})}$ such that $d(\mathbf{k})=v(\mathbf{k})\cdot
b(\mathbf{k}).$
\end{itemize}

It is clear that $b(\mathbf{k})=b(\boldsymbol{\lambda}).$ We define
$\zeta(\mathbf{k}):=\phi(\mathbf{k})\cdot v(\mathbf{k}).$ Then, comparing it
with
\[
b(\mathbf{k})=v(\mathbf{k})^{-1}\cdot d(\mathbf{k})=d(\mathbf{k})\cdot
v(\mathbf{k}),
\]
we have $\zeta(\mathbf{k})=\zeta(\boldsymbol{\lambda})$. In the sequel, 
we will use the notation $b(\boldsymbol{\lambda})$ and $\zeta(\boldsymbol{\lambda})$. 
From the definitions, we have
\[
\zeta(\boldsymbol{\lambda})=a(\mathbf{k})\cdot u(\mathbf{k})v(\mathbf{k})+
e(m(\mathbf{k})\cdot v(\mathbf{k})),
\]
which shows that $\zeta(\boldsymbol{\lambda})$ belongs to $a(\mathbf{k})W$.

\begin{example}
With $\mathbf{k}=(3,1,0,-2,-4,-6,-7)$, $e=2$ and $l=3,$ we obtain
\begin{align*}
c(\mathbf{k})  &  =(1,1,2,2,2,2,1),\\
d(\mathbf{k})  &  =(2,1,3,2,1,3,3),\\
m(\mathbf{k})  &  =(0,0,-1,-1,-1,-2,-2),\\
\phi(\mathbf{k})  &  =(1,1,0,0,0,-2,-3),\\
a(\mathbf{k})  &  =(1,1,1,2,2,2,2),\\
b(\mathbf{k})  &  =b(\boldsymbol{\lambda})=(3,3,3,2,2,1,1),\\
\zeta(\mathbf{k})  &  =\zeta(\boldsymbol{\lambda})=(0,-2,-3,1,0,1,0).
\end{align*}
\end{example}

\subsection{Ordering multipartitions}

Now we introduce the dominance order in a general setting. Let $k\in
\mathbb{N}$ and $\mathbf{u}=(u_{1},\ldots,u_{k})\in\mathbb{Q}^{k}%
,\mathbf{v}=(v_{1},\ldots,v_{k})\in\mathbb{Q}^{k}$. Then, we write
$\mathbf{u}\rhd\mathbf{v}$ if $\mathbf{u}\neq\mathbf{v}$ and
\[
\sum_{s=1}^{a}u_{s}\geq\sum_{s=1}^{a}v_{s},\text{ for }a=1,\ldots,k.
\]

We fix a decreasing sequence $1>\alpha_{1}>\alpha_{2}>\ldots>\alpha_{l}>0$ of
rational numbers. Then, for each ${\boldsymbol{\lambda}}\in\Pi_{l,n}$, we read
the rational numbers
\[
\lambda_{j}^{(i)}-j+s_{i}-\alpha_{i},\;\text{for $j=1,\ldots,n+s_{i}$ and
$i=1,\ldots,l$,}%
\]
in decreasing order and denote the resulting sequence by $\gamma
({\boldsymbol{\lambda}})\in\mathbb{Q}^{k}$ where $k=\sum_{i=1}^{l} s_{i}+nl$.
Note that one can recover ${\boldsymbol{\lambda}}$ from 
$\gamma({\boldsymbol{\lambda}})=(\gamma_{1},\ldots,\gamma_{k})$. Hence, 
if $\gamma({\boldsymbol{\lambda}})=\gamma({\boldsymbol{\mu}})$ then ${\boldsymbol
{\lambda}}={\boldsymbol{\mu}}$. 
This follows from the fact that for all $i\in [1,l]$, the set 
$$\{ \gamma_k-s_{i}+\alpha_{i}\ |\ \gamma_k-[\gamma_k]=\alpha_i\}$$
is the set of $\beta$-numbers of $\lambda^i$. 


\begin{definition}
Let $\boldsymbol{\lambda},\boldsymbol{\mu}\in\Pi_{l,n}.$ Then we write
${\boldsymbol{\lambda}}\succ{\boldsymbol{\mu}}$ if $\gamma({\boldsymbol
{\lambda}})\rhd\gamma({\boldsymbol{\mu}}).$
\end{definition}

One can check that this defines a partial order which depends on the choice of
$\alpha$ but does not depend on $e$. This is a crucial remark in view of the
following result.

\begin{theorem}
\ \label{TH_Uglov}

\begin{enumerate}
\item  For each ${\boldsymbol{\lambda}}\in\Pi_{l,n}$, there exist polynomials
$\Delta_{{\boldsymbol{\lambda}},{\boldsymbol{\mu}}}^{e}(v)\in\mathbb{Z}[v]$,
for ${\boldsymbol{\mu}}\in\Pi_{l,n}$, such that we have the unitriangular
expansion
\[
G_{e}({\boldsymbol{\lambda},}\mathbf{s})={\boldsymbol{\lambda}+}%
\sum_{\boldsymbol{\lambda}\succ\boldsymbol{\mu}}\Delta_{{\boldsymbol{\lambda}%
},{\boldsymbol{\mu}}}^{e}(v){\boldsymbol{\mu}}.
\]

\item  For each ${\boldsymbol{\lambda}}\in\Pi_{l,n}$, there exist polynomials
$\Delta_{{\boldsymbol{\lambda}},{\boldsymbol{\mu}}}^{\infty}(v)\in
\mathbb{Z}[v]$, for ${\boldsymbol{\mu}}\in\Pi_{l,n}$, such that we have the
unitriangular expansion
\[
G_{\infty}({\boldsymbol{\lambda},}\mathbf{s})={\boldsymbol{\lambda}+}%
\sum_{{\boldsymbol{\lambda}}\succ{\boldsymbol{\mu}}}\Delta_{{\boldsymbol
{\lambda}},{\boldsymbol{\mu}}}^{\infty}(v){\boldsymbol{\mu}}.
\]

\item  For each pair $({\boldsymbol{\lambda}},{\boldsymbol{\mu}})\in\Pi
_{l,n}\times\Pi_{l,n}$, $\Delta_{{\boldsymbol{\lambda}},{\boldsymbol{\mu}}%
}^{e}(v)$ and $\Delta_{{\boldsymbol{\lambda}},{\boldsymbol{\mu}}}^{\infty}(v)$
are expressed by certain parabolic Kazhdan-Lusztig polynomials (see Section
\ref{Sec_pos}). In particular, they are polynomials with nonnegative integer coefficients.
\end{enumerate}
\end{theorem}

\begin{proof}
We prove $(1)$ and $(2)$ by the arguments which are similar to those used
in \cite{J2}.
As in \cite{U}, it suffices to show that the matrix of the bar-involution is
unitriangular with respect to  $\succ$.  Then the results
immediately follow from the characterization of the canonical basis.
We recall the bar-involution on the space $\bigwedge^{s+\infty/2} V_{e,l}$,
which is defined in \cite{U}, where $s=s_1+\ldots+ s_l$. The space
$\bigwedge^{s+\infty/2} V_{e,l}$ is
the $\mathbb{Q}(v)$-vector space spanned by the semi-infinite monomials
$$u_{\bf k}=u_{k_1}\wedge u_{k_2} \wedge\ldots, $$
where $k_i\in\mathbb{Z}$, for all $i\geq1$, and $k_i=s-i+1$ if $i>>0$.
Its basis is given by the \emph{ordered} monomials (i.e. the monomials with
decreasing indices $k_1>k_2 >\ldots$) because any monomial may be
expressed as
a linear combination of ordered monomials by ``straightening relations''
in \cite[Prop. 3.16]{U}.
Now, the procedure in \S\ref{subsec_bij} yields a bijection $\tau
_l$ from the set
of ordered
monomials to the set of pairs $(\ulambda,{\bf s})$ such that $\ulambda\in
\Pi_{l,n}$ and
${\bf s}=(s_1,\ldots,s_l)$ with $s=s_1+\ldots+ s_l$. This allows us
to identify the
space $\bigwedge^{s+\infty/2} V_{e,l}$ with
$\bigoplus_{s_1+\ldots+s_l=s} \mathcal{F}^{\bf s}_e$.
Let  $u_{\bf k}$ be a semi-infinite (possibly non ordered) monomial. Let
$u_{\widetilde{{\bf k}}}$  be the monomial obtained from
$u_{\bf k}$  by reordering the $k_i$'s in strictly decreasing order.
The bijection $\tau_l$
then allows us to associate a pair $(\ulambda,{\bf s})$ with
$u_{\widetilde{{\bf k}}}$  such that $\ulambda\in\Pi_{l,n}$ and ${\bf
s}=(s_1,\ldots,s_l)$. We define a map $\pi$ on the set of semi-infinite monimials by
$$\pi(u_{\bf k})=(\ulambda,{\bf s}).$$
In particular, $\tau_l$ and $\pi$ coincide on the set of ordered monomials.
Uglov defined a bar-involution
on $\bigwedge^{s+\infty/2} V_{e,l}$ as follows : for all semi-infinite
ordered monomials
$u_{\bf k}$, we define
$$\overline{u_{\bf k}}:=v^t u_{k_r}\wedge u_{k_{r-1}}\wedge\ldots
\wedge u_{k_1} \wedge
u_{k_{r+1}}\wedge u_{k_{r+2}}\wedge\ldots$$
where $t$ is a certain integer (see \cite[\S3.4]{U} for its explicit
definition) and $r$
is a sufficiently large integer. Hence, to compute $\overline{\ulambda}$ in
$\mathcal{F}^{\bf s}_e$, we set $u_{\bf k}=\tau_l^{-1} (\ulambda,{\bf s})$
and use the
straightening relations to expand $\overline{u_{\bf k}}$
on the basis of the ordered monomials, and apply $\pi$ to obtain
the expression of $\overline{\ulambda}$ as a linear combination of
$l$-partitions.
We note that $\ulambda$ appears with coefficient $1$ by \cite[Rk. 3.24]{U}.
Let $u_{\bf p}$ be an arbitrary semi-infinite monomial and assume that
this is non ordered. Then there exists $i\in\mathbb{N}$
such that $k_i<k_{i+1}$. The straightening relations then show how to
express $u_{\bf p}$ in terms of
semi-infinite monomials $u_{{\bf p}'}$ with $p_i '>p_{i+1}'$.   Let us
denote $\pi(u_{\bf p})= (\ulambda,{\bf s})$
and $\pi(u_{{\bf p}'})= (\ulambda',{\bf s}')$.
A study of the straightening relations shows that we have  ${\bf
s}={\bf s}'$ and
that $\ulambda$ and
$\ulambda'$ are both obtained from the same $l$-partition
$\boldsymbol{\nu}$ by
adding a ribbon of fixed size $m$ (see \cite[\S4.2]{J2}).
We consider the set :
\[\{\beta_1,\ldots, \beta_h\}:=\left\{
\nu_{j}^{(i)}-j+s_{i}-\alpha_i,\;\text{for $j=1,\ldots,n+s_i$ and
$i=1,\ldots,l$}\right\}
\]
Then    there exists $a$ and $b$ such that   $\gamma(\ulambda)$ is the
sequence obtained
by reordering the elements  of
$\{\beta_1,\ldots, \beta_h\}\setminus\{\beta_a\} \cup\{\beta_a+m\}$
in decreasing
order and  $\gamma(\ulambda')$ is the sequence obtained by reordering the
elements      $\{\beta_1,\ldots, \beta_h\}\setminus\{\beta_b\} \cup
\{\beta_b+m\}$
in decreasing
order. Then, mimicking the argument in \cite[p.581-583]{J2}, one can prove
by a careful study of the straightening rules  that :
$$\beta_a >\beta_b.$$
This implies that $\ulambda\succ\ulambda'$. In particular, all the
ordered
monomials          $u_{{\bf k}'}$ which appear in the expansion of
$\overline{u_{\bf k}}$ satisfy
the following property : if
$ \pi(u_{\bf k}')= (\ulambda',{\bf s})$         then $\ulambda\succ
\ulambda'$. This proves $(1)$ and $(2)$.  The third part is a result of Uglov
\cite{U}. Uglov proved
that the coefficients $\Delta_{{\boldsymbol
{\lambda}},{\boldsymbol{\mu}}}^{e}(v)$ are expressed by parabolic
Kazhdan-Lusztig polynomials as we will see in Section \ref{Sec_pos}. By
results of
Kashiwara and Tanisaki \cite{KT}, this implies that they have nonnegative
integer coefficients.
\end{proof} 

\begin{remark}
The order $\succ$ does not coincide with the partial order used by Uglov in
\cite{U}. His partial order depends on $e$, so that he could not use a common
partial order in the statements (1) and (2) of Theorem \ref{TH_Uglov}. On the
other hand, we have used the common partial order $\succ$ there.
\end{remark}

As a direct consequence, we have the following theorem :

\begin{theorem}
\label{Th_main}For each ${\boldsymbol{\lambda}}\in\Pi_{l}$, we may expand
$G_{e}(\boldsymbol{\lambda},\mathbf{s})$ as follows.
\begin{equation}
G_{e}({\boldsymbol{\lambda},\mathbf{s})=}\sum_{{\boldsymbol{\nu}\in}\Pi_{l}%
}d_{{\boldsymbol{\lambda},\boldsymbol{\nu}}}(v)G_{\infty}({\boldsymbol{\nu
},\mathbf{s}}) \label{main_dec2}%
\end{equation}
where :

\begin{itemize}
\item $d_{{\boldsymbol{\lambda}},{\boldsymbol{\lambda}}}(v)=1$,

\item $d_{{\boldsymbol{\lambda}},{\boldsymbol{\nu}}}(v)\in v\mathbb{Z}%
[v]\text{ if }{\boldsymbol{\lambda}}\neq{\boldsymbol{\nu}}$

\item $d_{{\boldsymbol{\lambda}},{\boldsymbol{\nu}}}(v)\neq0$ only if
${\boldsymbol{\lambda}}\succeq{\boldsymbol{\nu}}$.
\end{itemize}
\end{theorem}

\begin{proof}
It follows from Theorem \ref{TH_Uglov}(1) and (2).
\end{proof}

\begin{corollary}
\label{cor_Hw}For ${\boldsymbol{\lambda}}\in B_{e}(\mathbf{s})$, the formula
(\ref{main_dec2}) has the form
\begin{equation}
G_{e}({\boldsymbol{\lambda},\mathbf{s})=}\sum_{{\boldsymbol{\nu}\in}%
B_{\infty}({\mathbf{s}})}d_{{\boldsymbol{\lambda},\boldsymbol{\nu}}%
}(v)G_{\infty}({\boldsymbol{\nu},\mathbf{s}}). \label{dec_HW}%
\end{equation}
\end{corollary}

\begin{proof}
We have already observed that $V_{e}({\mathbf{s}})$ may be regarded as a
$\mathcal{U}_{v}^{\prime}(\widehat{\mathfrak{sl}_{e}})$-submodule
of $V_{\infty}({\mathbf{s}})$ which shares
the common highest weight vector $\mathbf{\emptyset}$. Thus,
we may expand $G_{e}({\boldsymbol{\lambda},
\mathbf{s})\in}\mathcal{G}_{e}^{\circ}(\mathbf{s})$ on the basis
$\mathcal{G}_{\infty}^{\circ}(\mathbf{s})\subset\mathcal{G}_{\infty}%
(\mathbf{s}),$
and Theorem \ref{Th_main} implies (\ref{dec_HW}).
\end{proof}

\begin{definition}
We define
\begin{align*}
\Delta_{e}(v)  &  =(\Delta_{{\boldsymbol{\lambda}},{\boldsymbol{\mu}}}%
^{e}(v))_{{\boldsymbol{\lambda}}\in\Pi_{l},{\boldsymbol{\mu}}\in\Pi_{l}},\\
\Delta_{\infty}(v)  &  =(\Delta_{{\boldsymbol{\lambda}},{\boldsymbol{\mu}}%
}^{\infty}(v))_{{\boldsymbol{\lambda}}\in\Pi_{l},{\boldsymbol{\mu}}\in\Pi_{l}%
},\\
\Delta_{\infty}^{e}(v)  &  =(d_{{\boldsymbol{\lambda}},{\boldsymbol{\nu}}%
}(v))_{{\boldsymbol{\lambda}}\in\Pi_{l},{\boldsymbol{\nu}}\in\Pi_{l}}.
\end{align*}
They depend on $\mathbf{s}$. Then, we have
\[
\Delta_{e}(v)=\Delta_{\infty}(v)\Delta_{\infty}^{e}(v).
\]
\end{definition}

\noindent We also define the following submatrices
\begin{align*}
D_{e}(v)  &  =(\Delta_{{\boldsymbol{\lambda}},{\boldsymbol{\mu}}}%
^{e}(v))_{{\boldsymbol{\lambda}}\in B_{e}(\mathbf{s}),{\boldsymbol{\mu}}\in
B_{\infty}(\mathbf{s})},\\
D_{\infty}(v)  &  =(\Delta_{{\boldsymbol{\lambda}},{\boldsymbol{\mu}}}%
^{\infty}(v))_{{\boldsymbol{\lambda}}\in B_{e}(\mathbf{s}),{\boldsymbol{\mu}%
}\in B_{\infty}(\mathbf{s})},\\
D_{\infty}^{e}(v)  &  =(d_{{\boldsymbol{\lambda}},{\boldsymbol{\nu}}%
}(v))_{{\boldsymbol{\lambda}}\in B_{e}(\mathbf{s}),{\boldsymbol{\nu}}\in
B_{\infty}(\mathbf{s})}.
\end{align*}
Then we have $D_{e}(v)=D_{\infty}(v)D_{\infty}^{e}(v).$

\begin{remark}
If $l=1$, then the matrix $D_{\infty}(v)$ is the identity and $D_{\infty}%
^{e}(v)=D_{e}(v)$.
\end{remark}

\section{Computation of $\Delta_{\infty}^{e}(v)$ and $D_{\infty}^{e}(v)$}

Before proceeding further, we explain algorithmic aspects for computing
$\Delta_{\infty}^{e}(v)$ and $D_{\infty}^{e}(v)$. As $\Delta_{\infty}%
^{e}(v)=\Delta_{\infty}^{-1}(v).\Delta_{e}(v)$, we start with computing
$\Delta_{\infty}(v)$ and $\Delta_{e}(v)$. Two algorithms are already proposed:
one by Uglov and the other by Yvonne. Both use a natural embedding of the Fock
spaces $\mathcal{F} _{e}^{\mathbf{s}}$ into the space of semi-infinite wedge
products and compute the canonical bases $\mathcal{G}_{e}(\mathbf{s})$ and
$\mathcal{G}_{\infty}(\mathbf{s})$.

The algorithm described by Uglov \cite{U} needs steps to compute straightening
laws of the wedge products, which soon starts to require enormous resources
for the computation. It occurs especially in the case when the differences
between two consecutive entries of $\mathbf{s}$ are large.

Yvonne's algorithm \cite{XY2} is much more efficient but it requires subtle
computation related to the commutation relations of ${\mathcal{U}_{v}%
(\widehat{\mathfrak{sl}_{e}})}\otimes\mathcal{H} \otimes{\mathcal{U}_{-v^{-1}%
}(\widehat{\mathfrak{sl}_{l}})}$ on the space of semi-infinite wedge products,
where $\mathcal{H}$ is the Heisenberg algebra. We do not pursue this direction
and refer to \cite{XY2} for complete description of this algorithm.

Once $\mathcal{G}_{e}(\mathbf{s})$ and $\mathcal{G}_{\infty}(\mathbf{s})$ are
computed, we can efficiently compute $\Delta_{\infty}^{e}(v)$ from them: see
\S\ref{sub_gene_algo} below.

The computation of $D_{\infty}^{e}(v)$ is easier. One may compute it directly
from the canonical bases $\mathcal{G}_{e}^{\circ}(\mathbf{s})$ and
$\mathcal{G}_{\infty}^{\circ}(\mathbf{s})$ and we may compute the canonical
bases by the algorithms proposed in \cite{LLT} or \cite{J}. The algorithm
given in \cite{J} was originally suited for multicharges $\mathbf{s}$ such
that $0\leq s_{1}\leq s_{2}\leq\cdots\leq s_{l}<e.$ However, we will see in
\S\ref{algo} that it also computes the canonical bases $\mathcal{G}_{e}%
^{\circ}(\mathbf{s})$ and $\mathcal{G}_{\infty}^{\circ}(\mathbf{s})$ (and thus
the matrix $D_{\infty}^{e}(v)$) for arbitrary multicharge $\mathbf{s}$.
Observe that this only uses ${\mathcal{U}_{v}(\mathfrak{g})}$-module structure
of the Fock space.

\subsection{A general procedure}

\label{sub_gene_algo}

Assume that we have computed the canonical bases $\mathcal{G}_{e}(\mathbf{s})
$ and $\mathcal{G}_{\infty}(\mathbf{s})$. Using the unitriangularity of the
decomposition matrices, one can obtain $\Delta_{\infty}^{e}(v)$ directly from
the relation $\Delta_{\infty} ^{e}(v)=\Delta_{\infty}^{-1}(v).\Delta_{e}(v).$
This can be done efficiently by applying the procedure below.

\begin{enumerate}
\item  Let ${\boldsymbol{\lambda}}\in\Pi_{l,n}$. We know by Theorem
\ref{Th_main} that $G_{e}({\boldsymbol{\lambda},}\mathbf{s})$ may be expanded
on $\mathcal{G}_{\infty}(\mathbf{s})$. We denote
\[
\Lambda(\boldsymbol{\lambda}):=\{\boldsymbol{\nu}\in\Pi_{l,n}\mid
d_{{\boldsymbol{\lambda}},{\boldsymbol{\nu}}}(v)\neq0\}.
\]
Our aim is to find the members of $\Lambda(\boldsymbol{\lambda})$, and
determine $d_{{\boldsymbol{\lambda}},{\boldsymbol{\nu}}}(v)$ when
${\boldsymbol{\nu}}$ is a member. Set ${\boldsymbol{\lambda}}^{0}%
:=\boldsymbol{\lambda}.$ Then ${\boldsymbol{\lambda}}^{0}$ is a member and
$d_{{\boldsymbol{\lambda}},{\boldsymbol{\lambda}}^{0}}(v)=1$.

\item  Let $k\in\mathbb{N}$. Suppose that we already know $k$ members
${\boldsymbol{\lambda}}^{0},\dots,{\boldsymbol{\lambda}}^{k-1}$ of
$\Lambda(\boldsymbol{\lambda})$ and the polynomials $d_{{\boldsymbol{\lambda}%
},{\boldsymbol{\lambda}}^{i}}(v)$, for $i=0,\dots,k-1.$ Then, we expand
\[
G_{e}({\boldsymbol{\lambda},}\mathbf{s})-\sum_{i=0}^{k-1}d_{{\boldsymbol
{\lambda}},{\boldsymbol{\lambda}}^{i}}(v)G_{\infty}({\boldsymbol{\lambda}%
^{i},}\mathbf{s})
\]
into linear combination of the standard basis of $l$-partitions and write
\[
\sum_{\boldsymbol{\nu}\in\Lambda(\boldsymbol{\lambda})\setminus\{{\boldsymbol
{\lambda}}^{0},\dots,{\boldsymbol{\lambda}}^{k-1}\}}d_{{\boldsymbol{\lambda}%
},{\boldsymbol{\nu}}}(v)G_{\infty}({\boldsymbol{\nu},}\mathbf{s}%
)=\sum_{\boldsymbol{\mu}\in\Pi_{l,n}}b_{\boldsymbol{\mu}}(v){\boldsymbol{\mu}%
}.
\]
We have $b_{\boldsymbol{\mu}}(v)\in\mathbb{Z}[v]$ by Theorem \ref{TH_Uglov}.
If the right hand side is zero, we are done. Otherwise, let ${\boldsymbol
{\lambda}}^{k}$ be a maximal $l$-partition in $\{\boldsymbol{\mu}\in\Pi
_{l,n}\mid b_{\boldsymbol{\mu}}(v)\neq0\},$ with respect to the partial order
$\succ.$

\item  Consider ${\boldsymbol{\nu}}\in\Lambda(\boldsymbol{\lambda}%
)\setminus\{{\boldsymbol{\lambda}}^{0},\dots,{\boldsymbol{\lambda}}^{k-1}\}$
which satisfies ${\boldsymbol{\nu}}\succ{\boldsymbol{\lambda}}^{k}.$ If such
${\boldsymbol{\nu}}$ does not exist, then we have
\[
{\boldsymbol{\lambda}}^{k}\in\Lambda(\boldsymbol{\lambda})\setminus
\{{\boldsymbol{\lambda}}^{0},\dots,{\boldsymbol{\lambda}}^{k-1}\}.
\]
Otherwise let ${\boldsymbol{\nu}}^{k}$ be maximal among them. If
${\boldsymbol{\nu}}^{k}$ appears in $G_{\infty}({\boldsymbol{\nu},}%
\mathbf{s})$, for $\boldsymbol{\nu}\in\Lambda(\boldsymbol{\lambda}%
)\setminus\{{\boldsymbol{\lambda}}^{0},\dots,{\boldsymbol{\lambda}}^{k-1}\},$
then $\boldsymbol{\nu}\succeq{\boldsymbol{\nu}}^{k}\succ{\boldsymbol{\lambda}%
}^{k},$ so that the maximality implies $\boldsymbol{\nu}={\boldsymbol{\nu}%
}^{k}.$ Since ${\boldsymbol{\nu}}^{k}$ appears in $G_{\infty}({{\boldsymbol
{\nu}}^{k},}\mathbf{s})$, it follows that $b_{{\boldsymbol{\nu}}^{k}}(v)\neq
0$, which is impossible by the maximality of ${\boldsymbol{\lambda}}^{k}$ and
${\boldsymbol{\nu}}^{k}\succ{\boldsymbol{\lambda}}^{k}.$ Hence, ${\boldsymbol
{\lambda}}^{k}$ is a maximal element of $\Lambda(\boldsymbol{\lambda
})\setminus\{{\boldsymbol{\lambda}}^{0},\dots,{\boldsymbol{\lambda}}^{k-1}\}.$
Therefore, ${\boldsymbol{\lambda}}^{k}$ does not appear in $G_{\infty
}({\boldsymbol{\nu},}\mathbf{s})$, for $\boldsymbol{\nu}\in\Lambda
(\boldsymbol{\lambda})\setminus\{{\boldsymbol{\lambda}}^{0},\dots
,{\boldsymbol{\lambda}}^{k}\},$ and it follows that $d_{{\boldsymbol{\lambda}%
},{\boldsymbol{\lambda}}^{k}}(v)=b_{{\boldsymbol{\lambda}}^{k}}(v).$

\item  We increment $k$ and go to (2).
\end{enumerate}

\subsection{The computation of $\mathcal{G}_{e}^{\circ}(\mathbf{s})$ and
$\mathcal{G}_{\infty}^{\circ}(\mathbf{s})$}

\label{algo}

Let $e\in\mathbb{Z}_{\geq2}\cup\{\infty\}$. Assume first that $0\leq s_{1}\leq
s_{2}\leq\cdots\leq s_{l}<e$. It is proved in \cite{LLT} and \cite{J} that one
may construct a sequence of elements in $\mathbb{Z}/e\mathbb{Z}$
\begin{equation}
\underbrace{k_{1},\cdots,k_{1}}_{u_{1}},\underbrace{k_{2},\cdots,k_{2}}
_{u_{2}},\cdots,\underbrace{k_{s},\cdots,k_{s}}_{u_{s}} , \label{def_sequence}%
\end{equation}
for each ${\boldsymbol{\lambda}}\in B_{e}(\mathbf{s})$, such that if we
define
\[
A_{e}({\boldsymbol{\lambda},}\mathbf{s}):=f_{k_{1}}^{(u_{1})}\cdots f_{k_{s}
}^{(u_{s})}.\mathbf{\emptyset}\in V_{e}({\mathbf{s}})
\]
then
\[
\mathcal{A}_{e}(\mathbf{s})=\left\{  A_{e}({\boldsymbol{\lambda},}\mathbf{s}
)\ |\ {\boldsymbol{\lambda}}\in B_{e}(\mathbf{s})\right\}
\]
is a basis of $V_{e}(\mathbf{s})$. It is easy to obtain the coefficients
$\gamma_{{\boldsymbol{\lambda}},{\boldsymbol{\mu}}}(v)\in\mathbb{Z}[v,v^{-1}]$
in the expansion
\begin{equation}
G_{e}({\boldsymbol{\lambda},\mathbf{s}})=\sum_{{\boldsymbol{\mu}}\in
B_{e}(\mathbf{s})}\gamma_{{\boldsymbol{\lambda}},{\boldsymbol{\mu}}}
(v)A_{e}({\boldsymbol{\mu},\mathbf{s}}) . \label{BC_on_A}%
\end{equation}

\bigskip

When $e\in\mathbb{Z}_{\geq2}$, we have seen in \S\ref{subsec_bij} that there
is an action of the (extended) affine symmetric group $W$ on $\mathbb{Z}^{l}$
such that
\[
\mathcal{B}^{l}:=\left\{  (s_{1},\ldots,s_{l})\in\mathbb{Z}^{l}\ |\ 0\leq
s_{1}\leq\cdots\leq s_{l}<e\right\}
\]
is a fundamental domain for this action. Hence, for any $\mathbf{v:}
=(v_{1},\ldots,v_{l})\in\mathbb{Z}^{l}$, there exist $\mathbf{s}
:=(s_{1},\ldots,s_{l})\in\mathcal{B}^{l}$ and $w\in W$ such that
$\mathbf{v}=w.\mathbf{s}$. Since $\mathbf{v}$ and $\mathbf{s}$ yield the same
dominant weight, we have an isomorphism $\phi_{\mathbf{s},\mathbf{v}}$ from
$V_{e}(\mathbf{s})$ to $V_{e}(\mathbf{v})$. We can assume that $\phi
_{\mathbf{s},\mathbf{v}}(\mathbf{\emptyset})=\mathbf{\emptyset}.$ For each
${\boldsymbol{\lambda}}\in B_{e}(\mathbf{s}),$ we set
\[
A_{e}({\boldsymbol{\lambda},\mathbf{v}})=f_{k_{1}}^{(r_{1})}\cdots f_{k_{s}
}^{(r_{s})}.\mathbf{\emptyset}\in V_{e}(\mathbf{v}),
\]
where the pairs $(k_{a},r_{a})$ are defined by (\ref{def_sequence}). Then we
have $\phi_{\mathbf{s},\mathbf{v}}(A_{e} ({\boldsymbol{\lambda},\mathbf{s}%
}))=A_{e}({\boldsymbol{\lambda},\mathbf{v} }).$ By the uniqueness of the
crystal basis on $V_{e}(\mathbf{v})$ proved by Kashiwara, we also have
$\phi_{\mathbf{s},\mathbf{v}}(G_{e}({\boldsymbol{\lambda},\mathbf{s}}
))=G_{e}(\varphi_{\mathbf{s},\mathbf{v}}({\boldsymbol{\lambda}),\mathbf{v}}),
$ where $\varphi_{\mathbf{s},\mathbf{v}}$ is the crystal isomorphism from
$B_{e}(\mathbf{s})$ to $B_{e}(\mathbf{v})$ (see \cite{JL} for a combinatorial
description of $\varphi_{\mathbf{s},\mathbf{v}}$). By applying $\phi
_{\mathbf{s},\mathbf{v}}$ to (\ref{BC_on_A}), we obtain
\[
G_{e}({\boldsymbol{\nu},\mathbf{v}}) =\sum_{{\boldsymbol{\mu}}\in
B_{e}(\mathbf{s})} \gamma_{\varphi_{\mathbf{s},\mathbf{v}}^{-1}({\boldsymbol
{\nu}}),{\boldsymbol{\mu}}}(v) A_{e}({\boldsymbol{\mu},\mathbf{v}}),
\]
for ${\boldsymbol{\nu}\in} B_{e}(\mathbf{v}),$ and it follows that
\[
\mathcal{G}_{e}(\mathbf{v})=\left\{  \sum_{{\boldsymbol{\mu} }\in
B_{e}(\mathbf{s})}\gamma_{{\boldsymbol{\lambda}},{\boldsymbol{\mu}} }%
(v)A_{e}({\boldsymbol{\mu},\mathbf{v}})\ |\ {\boldsymbol{\lambda}}\in
B_{e}(\mathbf{s})\right\}  .
\]
Hence, the algorithms in \cite{LLT} and \cite{J2} compute the canonical basis
$\mathcal{G}_{e}(\mathbf{v})$ for any multicharge $\mathbf{v}=(v_{1}%
,\ldots,v_{l} )\in\mathbb{Z}^{l}$. Applying the general procedure in
\S\ref{algo} restricted to the canonical bases $\mathcal{G}_{e}^{\circ
}(\mathbf{v})$ and $\mathcal{G}_{\infty}^{\circ}(\mathbf{v})$, we may compute
$D_{\infty}^{e}(v)$.

\begin{remark}
Another algorithm is recently proposed by Fayers \cite{Fa} for computing the
canonical basis of the highest weight ${\mathcal{U}}_{v}^{\prime}%
(\widehat{\mathfrak{sl}_{e}})$-modules which is realized in the tensor product
of level one Fock spaces.
\end{remark}

\subsection{Example}

We set $e=2$, Then the matrix $D_{e}(v)$ of the canonical basis of the
$\mathcal{U}_{v}(\widehat{\mathfrak{sl}_{e}})$-module $V_{e}(0,0)$ is:
\[%
\begin{array}
[c]{c}%
(\emptyset,(3) )\\
((3),\emptyset)\\
((1),(2) )\\
((2),(1)\\
(\emptyset,(2.1))\\
((2.1),\emptyset)\\
((1),(1.1))\\
((1.1),(1))\\
(\emptyset, (1.1.1) )\\
((1.1.1),\emptyset)
\end{array}%
\begin{array}
[c]{c}%
\begin{array}
[c]{c}
\end{array}
\end{array}
\left(
\begin{array}
[c]{ccc}%
1 & . & .\\
v & . & .\\
v & 1 & .\\
v^{2} & v & .\\
. & . & 1\\
. & . & v\\
v & v^{2} & .\\
v^{2} & v^{3} & .\\
v^{2} & . & .\\
v^{3} & . & .
\end{array}
\right)
\]
where dots mean $0$ and each row is labeled by a $2$-partition of rank $3$.
The matrix $D_{\infty}(v)$ of the canonical basis of the $\mathcal{U}%
_{v}({\mathfrak{sl}_{\infty}})$-module $V_{\infty}{(0,0)}$ is:
\[%
\begin{array}
[c]{c}%
(\emptyset,(3) )\\
((3),\emptyset)\\
((1),(2) )\\
((2),(1)\\
(\emptyset,(2.1))\\
((2.1),\emptyset)\\
((1),(1.1))\\
((1.1),(1))\\
(\emptyset, (1.1.1) )\\
((1.1.1),\emptyset)
\end{array}%
\begin{array}
[c]{c}%
\begin{array}
[c]{c}
\end{array}
\end{array}
\left(
\begin{array}
[c]{ccccc}%
1 & . & . & . & .\\
v & . & . & . & .\\
. & 1 & . & . & .\\
. & v & . & . & .\\
. & . & 1 & . & .\\
. & . & v & . & .\\
. & . & . & 1 & .\\
. & . & . & v & .\\
. & . & . & . & 1\\
. & . & . & . & v
\end{array}
\right)
\]
The matrix $D^{e}_{\infty}(v)$ obtained from our algorithm is:
\[
\left(
\begin{array}
[c]{ccc}%
1 & . & .\\
v & 1 & .\\
. & . & 1\\
v & v^{2} & .\\
v^{2} & . & .
\end{array}
\right)
\]
and one can check that we have
\[
D_{e}(v)=D_{\infty}(v).D_{\infty}^{e}(v).
\]

\section{Positivity of the coefficients in $d_{{\boldsymbol{\protect\lambda},
\boldsymbol{\protect\nu}}}(v)$}

\label{Sec_pos}

The aim of this section is to study the entries of the matrix $D^e_{\infty}(v)
$. The main result asserts that they are polynomials with nonnegative integer coefficients.

\subsection{Some notation on KL-polynomials}

Let $H$ be the extended affine Hecke algebra of the symmetric group $S_{r}$.
Namely, it is generated by $T_{1},\dots,T_{r-1}$ and $X^{\lambda}$, for
$\lambda\in\oplus_{i=1}^{r}\mathbb{Z}\epsilon_{i}$, such that the defining
relations are
\begin{align*}
(T_{i}-v^{-1})(T_{i}+v)  &  =0,\quad X^{\lambda}T_{i}=T_{i}X^{s_{i}\lambda
}+(v-v^{-1})\frac{X^{s_{i}\lambda}-X^{\lambda}}{1-X^{\alpha_{i}}}\\
X^{\lambda}X^{\mu}  &  =X^{\mu}X^{\lambda},\quad X^{\lambda}X^{-\lambda}=1
\end{align*}
and the Artin braid relations for $T_{1},\dots,T_{r-1}.$ The affine Hecke
algebra admits a canonical basis $\{C_{w}^{\prime}\mid w\in W\}$ such that
\[
C_{w}^{\prime}=v^{\ell(w)}\sum_{y\in W,y\leq w}%
P_{y,w}(v^{-2})v^{-\ell(y)}T_{y}%
\]
where $\leq$ is the Bruhat order on $W.$ We refer the reader to \cite{Ram} and
\cite{KT} for a detailed review on affine Hecke algebras, the definition of
the relevant length function and the Kazhdan-Lusztig basis. The polynomials
$P_{y,w}(v^{-2})$ are the affine KL-polynomials. They admit nonnegative
integer coefficients. We also recall the following property
\begin{equation}
P_{y,w}=P_{s_{i}y,w} \label{prop_basic_P}%
\end{equation}
for any $y<w$ in $W$ and $i=1,\ldots,r$ such that $s_{i}w<w.$

\subsection{Expression of the coefficients $\Delta_{{\boldsymbol
{
\protect
\lambda}},{\boldsymbol{\protect\mu}}}^{e}(v)$ in terms of KL-polynomials}

\label{Expression in KL}The aim of this paragraph is to recall Uglov's
construction of finite wedge product\cite{U} and to show in a simpler manner than \cite{U} that the entries
$\Delta_{{\boldsymbol{\lambda
}},{\boldsymbol{\mu}}}^{e}(v)$ are expressed in terms parabolic
Kazhdan-Lusztig polynomials.

We want to introduce the space of finite wedge products. Consider $a\in A^{r}
$ and $b\in B^{r}.$ We define $_{a}\!W,W_{b},w_{a},w_{b}$ as in \S
\ref{subsec_bij}. The subgroups $_{a}\!W$ and $W_{b}$ define parabolic
subalgebras $H_{a}$ and $H_{b}$ of the affine Hecke algebra $H$. If we denote
\[
J=\{i\mid1\leq i\leq r-1,b_{i}=b_{i+1}\},
\]
then the parabolic subgroup $W_{J}$ is nothing but the subgroup $W_{b}$. Let
$1_{a}^{+}$ (resp. $1_{b}^{-}$) be the right $H_{a}$-module (resp. left
$H_{b}$-module) such that
\begin{equation}%
\begin{array}
[c]{lll}%
1_{a}^{+}T_{i}=v^{-1}1_{a}^{+}\;\; & (s_{i}\in_{a}\!\!W) & \\
T_{i}1_{b}^{-}=-v1_{b}^{-}\;\; & (s_{i}\in W_{b}). &
\end{array}
\label{def_1(a,b)}%
\end{equation}
We define $\Lambda^{r}(a,b)=1_{a}^{+}%
\otimes_{H_{a}}H\otimes_{H_{b}}1_{b}^{-}.$ Then, the \emph{space of finite wedges}
$\Lambda^{r}$ is the direct sum of the $\Lambda^{r}(a,b)$, for $a\in A^{r}$
and $b\in B^{r}.$ We define the bar-involution on $\Lambda^{r}$ by
\[
\overline{1_{a}^{+}\otimes h\otimes1_{b}^{-}}=1_{a}^{+}\otimes{\overline{h}%
}\otimes1_{b}^{-}.
\]

\begin{definition}
Let $\xi\in P.$ Then, there are unique $a\in A^{r}$ and $x\in{^{a}W}$ such
that $\xi=ax.$ We denote this $x$ by $x(\xi)$.
\end{definition}

We say that $\xi$ is \emph{$J$-dominant} and write $\xi\in P_{b}^{++}$, if
$\xi_{i}>\xi_{i+1}$ whenever $b_{i}=b_{i+1}$. Similarly, we say that $\xi\in
P_{b}^{+}$ if $\xi_{i}\geq\xi_{i+1}$ whenever $b_{i}=b_{i+1}$. Note that
$\zeta(\boldsymbol{\lambda})\in P_{b}^{++}$, for $\boldsymbol{\lambda}\in
\Pi_{l}$. If $\xi\in P_{b}^{++}$, it follows by \cite[Prop. 3.8]{U} that
$x(\xi)s<x(\xi)$ in the Bruhat order, for any $s\in W_{b}$. So $x(\xi)w_{b}$
is the minimal length coset representative of $_{a}\!Wx(\xi)W_{b}$.

By \cite[Lem. 3.19, Prop. 3.20]{U}, the \emph{wedge basis} of $\Lambda
^{r}(a,b)$ is given by
\[
\{|\boldsymbol{\lambda}\rangle=1_{a}^{+}\otimes T_{x(\zeta(\boldsymbol
{\lambda}))w_{b}}\otimes1_{b}^{-}=(-v)^{-\ell(w_{b})}1_{a}^{+}\otimes
T_{x(\zeta(\boldsymbol{\lambda}))}\otimes1_{b}^{-}\mid\zeta(\boldsymbol
{\lambda})\in aW\}.
\]
Here we have written for short, $a=a(\mathbf{k})$ and
$b=b(\mathbf{k)=b(\boldsymbol{\lambda})}$ where $\mathbf{k}=\tau_{l}%
^{-1}(\boldsymbol{\lambda},\mathbf{s})$. We put $x=x(\zeta
(\boldsymbol{\lambda}))w_{b}$. Then, by the Kazhdan-Lusztig theory,%
\[
C_{w_{a}x}^{\prime}=v^{\ell(w_{a}x)}\sum_{y\in W}P_{y,w_{a}x}(v^{-2}%
)v^{-\ell(y)}T_{y}%
\]
is bar-invariant. As
\begin{equation}
W\simeq_{a}\!\!W\times\{x(\eta)\mid\eta\in aW\} \label{x(heta)}%
\end{equation}
we have
\begin{align*}
C_{w_{a}x}^{\prime}  &  =v^{\ell(w_{a}x)}\sum_{\eta\in aW}\sum_{u\in_{a}\!%
W}P_{ux(\eta),w_{a}x}(v^{-2})v^{-\ell(u)-\ell(x(\eta))}T_{u}T_{x(\eta)}\\
&  =v^{\ell(w_{a}x)}\sum_{\eta\in aW}\sum_{u\in_{a}\!W}P_{w_{a}x(\eta),w_{a}%
x}(v^{-2})v^{-\ell(u)-\ell(x(\eta))}T_{u}T_{x(\eta)}
\end{align*}
where the last equality is a consequence of (\ref{prop_basic_P}). Set
\[
C_{e}^{+}(\boldsymbol{\lambda})=\frac{v^{-\ell(w_{a})}}{\sum_{u\in_{a}\!%
W}v^{-2\ell(u)}}1_{a}^{+}\otimes C_{w_{a}x}^{\prime}\otimes1_{b}^{-}%
\]
where $x=x(\xi)w_{b}$ and $\xi=\zeta(\boldsymbol{\lambda}).$ Then, using
(\ref{def_1(a,b)}), we have that
\[
C_{e}^{+}(\boldsymbol{\lambda})=\sum_{\eta\in aW}v^{\ell(x)-\ell(x(\eta
))}P_{w_{a}x(\eta),w_{a}x}(v^{-2})1_{a}^{+}\otimes T_{x(\eta)}\otimes1_{b}^{-}%
\]
is bar-invariant. When $\eta$ admits repeated entries, one can verify that
$1_{a}^{+}\otimes T_{x(\eta)}\otimes1_{b}^{-}$ is equal to $0$. Here we refer
the reader to \S3.3 of \cite{U} for a detailed proof (which justifies the
terminology of Fock space used). Now, we rewrite
$C_{e}^{+}(\boldsymbol{\lambda})$ into the expression
\begin{multline*}
C_{e}^{+}(\boldsymbol{\lambda})=\\
\sum_{\eta\in aW\cap P_{b}^{++}}\sum_{u\in W_{b}}v^{\ell(x)-\ell(x(\eta
)w_{b}u)}P_{w_{a}x(\eta)w_{b}u,w_{a}x}(v^{-2})(-v)^{\ell(u)}1_{a}^{+}\otimes
T_{x(\eta)w_{b}}\otimes1_{b}^{-}.
\end{multline*}
Recall that
\[
P_{w_{a}x(\eta)w_{b},w_{a}x(\xi)w_{b}}^{J,-1}(v^{-2})=\sum_{u\in W_{b}%
}(-1)^{\ell(u)}P_{w_{a}x(\eta)w_{b}u,w_{a}x(\xi)w_{b}}(v^{-2})
\]
is a parabolic Kazhdan-Lusztig polynomial. These polynomials were introduced
by Deodhar \cite{Deo}. As $x=x(\xi)w_{b}$ and $v^{\ell(x)-\ell(x(\eta)w_{b}%
u)}=v^{\ell(x(\xi))-\ell(x(\eta))-\ell(u)}$, we have
\begin{equation}
C_{e}^{+}(\boldsymbol{\lambda})=\sum_{\eta\in aW\cap P_{b}^{++}}v^{\ell
(x(\xi))-\ell(x(\eta))}P_{w_{a}x(\eta)w_{b},w_{a}x(\xi)w_{b}}^{J,-1}%
(v^{-2})1_{a}^{+}\otimes T_{x(\eta)w_{b}}\otimes1_{b}^{-}. \label{CBWed}%
\end{equation}
It satisfies the defining properties of the plus
canonical basis introduced by Uglov in \cite{U}. Thus, we have recovered
Uglov's result Theorem \ref{TH_Uglov}(3). 
To be more precise, let ${\boldsymbol{\lambda}},{\boldsymbol{\mu}}\in\Pi
_{l,n}$. Choose $r\in\mathbb{N}$ as in \S\ \ref{subsec_bij}, and define
$\mathbf{k},\mathbf{l}\in\mathbb{Z}^{r}$ by
\[
\mathbf{k}=\tau_{l}^{-1}(\boldsymbol{\lambda},\mathbf{s})\;\text{ and
}\;\mathbf{l}=\tau_{l}^{-1}(\boldsymbol{\mu},\mathbf{s}).
\]
Define $a(\mathbf{k}),a(\mathbf{l})$ and $b(\boldsymbol{\lambda}%
),b(\boldsymbol{\mu
})$ as in \S\ \ref{subsec_bij}, and set $\xi=\zeta(\boldsymbol{\lambda}) $
and $\eta=\zeta(\boldsymbol{\mu}).$

\begin{theorem}
\label{Th_Uglov_dec}With the above notation, we have

\begin{enumerate}
\item  If $a(\mathbf{k})\neq a(\mathbf{l}),$ or $b(\boldsymbol{\lambda})\neq
b(\boldsymbol{\mu})$ then $\Delta_{{\boldsymbol{\lambda}},{\boldsymbol{\mu}}%
}^{e}(v)=0.$

\item  If $a(\mathbf{k})=a(\mathbf{l})=a\in A^{r}$ and $b(\boldsymbol{\lambda
})=b(\boldsymbol{\mu})=b\in B^{r},$ then
\begin{equation}
\Delta_{{\boldsymbol{\lambda}},{\boldsymbol{\mu}}}^{e}(v)=v^{\ell(x(\xi
))-\ell(x(\eta))}P_{w_{a}x(\eta)w_{b},w_{a}x(\xi)w_{b}}^{J,-1}(v^{-2}).
\label{decomposition}%
\end{equation}
\end{enumerate}
\end{theorem}

\subsection{Stabilization for $e=\infty$}

\label{sub-sec_stab}

Now we assume that $\mathbf{s}\in\mathbb{Z}^{l}$ and ${\boldsymbol{\lambda}%
\in\Pi}_{l}$ are fixed and we increase $e$. By Remark \ref{rq_inde}, we have
seen that for any $e^{\prime}>e$, one can choose $r$ such that $\xi
=\zeta(\boldsymbol{\lambda})$ coincide for $e$ and $e^{\prime}$. Since
$\mathbf{s}$ and ${\boldsymbol{\lambda}}$ are fixed, when $e^{\prime}$ is
sufficiently large, there exist $\tilde{x}(\xi)\in S_{r}$ and $\tilde
{a}=(\tilde{a}_{1},\dots,\tilde{a}_{r})$ such that
\begin{equation}
\tilde{a}_{1}\leq\cdots\leq\tilde{a}_{r},\;\;\tilde{x}(\xi)\in{^{\tilde{a}}%
}S_{r}\;\;\text{ and }\;\;\xi=\tilde{a}\tilde{x}(\xi). \label{xi}%
\end{equation}
This only means that we do not need translations by $e^{\prime}\mu$, for $\mu\in P$, to 
reach the fundamental domain when $e^{\prime}$ is sufficiently large. In the sequel, we refer to this
stabilization phenomenon as the $e=\infty$ case. By Remark \ref{rq_inde} we
have the following expression for the $e=\infty$ case :
\[
\Delta_{{\boldsymbol{\lambda}},{\boldsymbol{\mu}}}^{\infty}(v)=v^{\ell
(\tilde{x}(\zeta(\boldsymbol{\lambda)}))-\ell(\tilde{x}(\zeta(\boldsymbol
{\mu)}))}P_{w_{\tilde{a}}\tilde{x}(\zeta(\boldsymbol{\mu)})w_{b},w_{\tilde{a}%
}\tilde{x}(\zeta(\boldsymbol{\lambda)})w_{b}}^{J,-1}(v^{-2}),
\]
for ${\boldsymbol{\lambda}},{\boldsymbol{\mu}}\in\Pi_{l}.$ Moreover, one can
assume that $r$ is adjusted such that $b$ and $\xi=\zeta(\boldsymbol{\lambda})$ 
are the same for $e$ finite (fixed) and $e=\infty$. In particular, we have 
$\xi\in\tilde{a}S_{r}$, for ${\boldsymbol{\lambda}}\in
\Pi_{l}$, as before. Then, Theorem
\ref{Th_Uglov_dec}(2) implies that we may assume $\eta S_{r}=\tilde{a}S_{r}$ for {$\eta
=\zeta(\boldsymbol{\mu})$.}

Recall that {$\xi=\zeta(\boldsymbol{\lambda})$ and $\eta=\zeta(\boldsymbol
{\mu})$ belong to $P_{b}^{++}$. }Then Theorems \ref{Th_main} and
\ref{Th_Uglov_dec} imply that there exist polynomials
\[
d_{\gamma\xi}(v)\in\mathbb{Z}[v],\;\;\text{for $\gamma\in P_{b}^{++}$,}%
\]
such that
\begin{multline*}
v^{\ell(x(\xi))-\ell(x(\eta))}P_{w_{a}x(\eta)w_{b},w_{a}x(\xi)w_{b}}%
^{J,-1}(v^{-2})=\\
\sum_{\gamma\in\tilde{a}S_{r}\cap P_{b}^{++}}v^{\ell(\tilde{x}(\gamma
))-\ell(\tilde{x}(\eta))}P_{w_{\tilde{a}}\tilde{x}(\eta)w_{b},w_{\tilde{a}%
}\tilde{x}(\gamma)w_{b}}^{J,-1}(v^{-2})d_{\gamma\xi}(v).
\end{multline*}
Define a linear map $\psi:\Lambda(\tilde{a},b)\hookrightarrow\Lambda(a,b)$ by
\[
1_{\tilde{a}}^{+}\otimes T_{\tilde{x}(\xi)w_{b}}\otimes1_{b}^{-}\mapsto
1_{a}^{+}\otimes T_{x(\xi)w_{b}}\otimes1_{b}^{-}=1_{a}^{+}\otimes
T_{x(\tilde{a})}T_{\tilde{x}(\xi)w_{b}}\otimes1_{b}^{-}.
\]
Then, in view of (\ref{CBWed}), the above equality is equivalent to
\begin{equation}
C_{e}^{+}(\boldsymbol{\lambda})=\sum_{\tilde{a}\in aW\cap P^{-}}\sum
_{\gamma:=\zeta(\boldsymbol{\nu})\in\tilde{a}S_{r}\cap P_{b}^{++}}d_{\gamma
\xi}(v)\psi(C_{\infty}^{+}(\boldsymbol{\nu})), \label{refomulation}%
\end{equation}
where
\begin{equation}
C_{\infty}^{+}(\boldsymbol{\nu})=\frac{v^{-\ell(w_{\tilde{a}})}}{\sum_{u\in
_{\tilde{a}}\!W}v^{-2\ell(u)}}1_{\tilde{a}}^{+}\otimes C_{w_{\tilde{a}}\tilde
{x}(\zeta(\boldsymbol{\nu}))w_b}^{\prime}\otimes1_{b}^{-}. \label{BcInfinite}%
\end{equation}

\subsection{Proof of the positivity}

The idea of the proof is to expand $C_{e}^{+}({\boldsymbol{\lambda}})$ into a
linear combination of $\psi(C_{\infty}^{+}(\boldsymbol{\mu}))$ and compare it
with (\ref{refomulation}). The famous positivity result of the multiplicative
structure constants with respect to the Kazhdan-Lusztig basis and its
generalization in \cite{GH} then yields the desired positivity\footnote{One
purpose of \cite{GH} is to introduce LLT polynomials for general root systems.
Note that LLT polynomials for finite root systems other than type A had been
introduced independently in \cite{LEC}. It is interesting to compare the two
definitions.}. Recall the basis
\[
C_{w}^{\prime}=v^{\ell(w)}\sum_{y\in W}P_{y,w}(v^{-2})v^{-\ell(y)}T_{y}.
\]
For $y\in W$, we write $y=y^{\prime}y^{\prime\prime}$, where $y^{\prime\prime
}\in S_{r}$ and $y^{\prime}$ is the minimal length coset representative of
$yS_{r}$. Then we define
\begin{equation}
U_{y}=T_{y^{\prime}}C_{y^{\prime\prime}}^{\prime}. \label{def_U}%
\end{equation}
It is clear that we may write
\begin{equation}
C_{w}^{\prime}=\sum_{y\in W}A_{y,w}(v)U_{y}, \label{C'_on_U}%
\end{equation}
where $A_{y,w}(v)\in\mathbb{Z}[v,v^{-1}]$. By \cite[Cor. 3.9]{GH}, we have in
fact $A_{y,w}(v)\in\mathbb{N}[v,v^{-1}]$.

We write $y=ux(\gamma)$, for $u\in_{a}\!\!W$ and $\gamma\in aW$, by 
(\ref{x(heta)}). Then we have
\[
U_{y}=U_{ux(\gamma)}=T_{u}U_{x(\gamma)}
\]
and it implies that 
\[
T_iU_y=\begin{cases} U_{s_iy} \quad&(s_iy>y)\\
                     (v^{-1}-v)U_{y}+U_{s_{i}y} &(s_iy<y).\end{cases}
\]
Let $w=w_{a}x(\xi)w_{b}$ and $\xi=\zeta({\boldsymbol{\lambda}})$. 
As $s_{i}w<w$, for $s_{i}\in_{a}\!\!W$, we deduce
\begin{align*}
v^{-1}C_{w}^{\prime}  &  =T_{i}C_{w}^{\prime}=\sum_{s_{i}y>y}A_{y,w}%
(v)U_{s_{i}y}+\sum_{s_{i}y<y}A_{y,w}(v)\left(  (v^{-1}-v)U_{y}+U_{s_{i}%
y}\right) \\
&  =\sum_{s_{i}y<y}\left(  A_{s_{i}y,w}(v)+(v^{-1}-v)A_{y,w}(v)\right)
U_{y}+\sum_{s_{i}y>y}A_{s_{i}y,w}(v)U_{y}.
\end{align*}
Thus, $A_{s_{i}y,w}(v)=v^{-1}A_{y,w}(v)$ if $s_{i}y>y$, and it follows that 
\[
A_{y,w}(v)=v^{-\ell(u)}A_{x(\gamma),w}(v),\quad\text{for $y=ux(\gamma)$}.
\]
Therefore, we have
\begin{equation}
\left(  \sum_{u\in_{a}\!W}v^{-\ell(u)}T_{u}\right)  \left(  \sum_{\gamma\in
aW}A_{x(\gamma),w}(v)U_{x(\gamma)}\right)  =C_{w}^{\prime}. \label{C'_factor}%
\end{equation}
Hence, for any ${\boldsymbol{\lambda}}\in\Pi_{l},$ the plus canonical basis is
given by
\begin{align}
C_{e}^{+}({\boldsymbol{\lambda}})  &  =\frac{v^{-\ell(w_{a})}}{\sum_{u\in
_{a}\!W}v^{-2\ell(u)}}1_{a}^{+}\otimes C_{w}^{\prime}\otimes1_{b}^{-}%
\label{express_C+}\\
&  =\sum_{\gamma\in aW}v^{-\ell(w_{a})}A_{x(\gamma),w_{a}x(\xi)w_{b}}%
(v)1_{a}^{+}\otimes U_{x(\gamma)}\otimes1_{b}^{-}\nonumber\\
&  =\sum_{\tilde{a}\in aW\cap P^{-}}\sum_{z\in S_{r}}v^{-\ell(w_{a}%
)}A_{x(\tilde{a})z,w_{a}x(\xi)w_{b}}(v)1_{a}^{+}\otimes T_{x(\tilde{a})}%
C_{z}^{\prime}\otimes1_{b}^{-}\nonumber
\end{align}
where the second equality follows from $w=w_{a}x(\xi)w_{b}$,
(\ref{def_1(a,b)}) and (\ref{C'_factor}), the third from (\ref{def_U}). Note
that $_{\tilde{a}}\!W=S_{r}\cap x(\tilde{a})^{-1} {_{a}\!W}x(\tilde{a})$ by
$\tilde{a}=ax(\tilde{a})$. Then (\ref{def_1(a,b)}) allows us to write
\[
1_{a}^{+}\otimes T_{x(\tilde{a})}C_{z}^{\prime}\otimes1_{b}^{-}=\frac{1}%
{\sum_{u\in _{\tilde{a}}\!W}v^{-2\ell(u)}}1_{a}^{+}\otimes T_{x(\tilde{a})}%
(\sum_{u\in _{\tilde{a}}\!W}v^{-\ell(u)}T_{u})C_{z}^{\prime}\otimes1_{b}^{-}.
\]
As the left multiplication by $\sum_{u\in _{\tilde{a}}\!W}v^{-\ell(u)}T_{u}$
gives the subspace of dimension $|S_{r}|/|_{\tilde{a}}\!W|$ in the Hecke algebra
$H(S_{r})$, it has the basis $\{C_{w_{\tilde{a}}y}^{\prime}\mid y\in
_{\tilde{a}}\!\!W\backslash S_{r}\}$. By the positivity of the structure
constants, we may write
\[
(\sum_{u\in _{\tilde{a}}\!W}v^{-\ell(u)}T_{u})C_{z}^{\prime}=\sum_{y\in
_{\tilde{a}}\!W\backslash S_{r}}B_{y,z}(v)C_{w_{\tilde{a}}y}^{\prime}%
\]
where $B_{y,z}(v)\in\mathbb{N}[v,v^{-1}]$. Thus,
\[
1_{a}^{+}\otimes T_{x(\tilde{a})}C_{z}^{\prime}\otimes1_{b}^{-}=\sum
_{\gamma\in\tilde{a}S_{r}}B_{\tilde{x}(\gamma),z}(v)\frac{1}{\sum_{u\in
_{\tilde{a}}\!W}v^{-2\ell(u)}}1_{a}^{+}\otimes T_{x(\tilde{a})}C_{w_{\tilde{a}%
}\tilde{x}(\gamma)}^{\prime}\otimes1_{b}^{-}.
\]
For each $\gamma\in\tilde{a}S_{r}$, define
\begin{equation}
d_{\gamma,\xi}^{\prime}(v)=v^{-\ell(w_{a})}\sum_{z\in S_{r}}v^{\ell
(w_{\tilde{a}})}A_{x(\tilde{a})z,w_{a}x(\xi)w_{b}}(v)B_{\tilde{x}(\gamma
),z}(v). \label{def_d'}%
\end{equation}
Then, $d_{\gamma,\xi}^{\prime}(v)\in\mathbb{N}[v,v^{-1}]$ and we have 
\begin{multline*}
\sum_{z\in S_{r}}v^{-\ell(w_{a})}A_{x(\tilde{a})z,w_{a}x(\xi)w_{b}}%
(v)1_{a}^{+}\otimes T_{x(\tilde{a})}C_{z}^{\prime}\otimes1_{b}^{-}\\%
\begin{split}
&  =\sum_{\gamma\in\tilde{a}S_{r}}d_{\gamma,\xi}^{\prime}(v)\frac
{v^{-\ell(w_{\tilde{a}})}}{\sum_{u\in _{\tilde{a}}\!W}v^{-2\ell(u)}}1_{a}%
^{+}\otimes T_{x(\tilde{a})}C_{w_{\tilde{a}}\tilde{x}(\gamma)}^{\prime}%
\otimes1_{b}^{-}\\
&  =\sum_{\gamma\in\tilde{a}S_{r}\cap P_{b}^{+}}\sum_{t\in W_{b}}d_{\gamma
w_{b}t,\xi}^{\prime}(v)\frac{v^{-\ell(w_{\tilde{a}})}}{\sum_{u\in _{\tilde
{a}}\!W}v^{-2\ell(u)}}1_{a}^{+}\otimes T_{x(\tilde{a})}C_{w_{\tilde{a}}\tilde
{x}(\gamma)w_{b}t}^{\prime}\otimes1_{b}^{-}%
\end{split}
\end{multline*}
where we slightly abuse the notation by using the same index $\gamma$ in the
last two expressions. If $xs_{i}<x$, for some $s_{i}\in W_{b}$, then
\[
v^{-1}C_{x}^{\prime}\otimes1_{b}^{-}=C_{x}^{\prime}T_{i}\otimes1_{b}%
^{-}=-vC_{x}^{\prime}\otimes1_{b}^{-}%
\]
and $C_{x}^{\prime}\otimes1_{b}^{-}=0$. Thus, we have in fact%

\begin{multline*}
\sum_{z\in S_{r}}v^{-\ell(w_{a})}A_{x(\tilde{a})z,w_{a}x(\xi)w_{b}}%
(v)1_{a}^{+}\otimes T_{x(\tilde{a})}C_{z}^{\prime}\otimes1_{b}^{-}=\\
\sum_{\gamma\in\tilde{a}S_{r}\cap P_{b}^{+}}d_{\gamma w_{b},\xi}^{\prime
}(v)\frac{v^{-\ell(w_{\tilde{a}})}}{\sum_{u\in _{\tilde{a}}\!W}v^{-2\ell(u)}%
}1_{a}^{+}\otimes T_{x(\tilde{a})}C_{w_{\tilde{a}}\tilde{x}(\gamma)w_{b}%
}^{\prime}\otimes1_{b}^{-}.
\end{multline*}
By using the last expression in (\ref{express_C+}) , we derive
\[
C_{e}^{+}({\boldsymbol{\lambda}})=\sum_{\tilde{a}\in aW\cap P^{-}}\sum
_{\gamma\in\tilde{a}S_{r}\cap P_{b}^{++}}d_{\gamma w_{b},\xi}^{\prime}%
(v)\frac{v^{-\ell(w_{\tilde{a}})}}{\sum_{u\in _{\tilde{a}}\!W}v^{-2\ell(u)}%
}1_{a}^{+}\otimes T_{x(\tilde{a})}C_{w_{\tilde{a}}\tilde{x}(\gamma)w_{b}%
}^{\prime}\otimes1_{b}^{-}.
\]
By using (\ref{BcInfinite}), this can also be rewritten
\[
C_{e}^{+}({\boldsymbol{\lambda}})=\sum_{\tilde{a}\in aW\cap P^{-}}\sum
_{\gamma=\zeta(\boldsymbol{\nu})\in\tilde{a}S_{r}\cap P_{b}^{++}}d_{\gamma w_{b},\xi}^{\prime}%
(v)\psi(C_{\infty}^{+}(\boldsymbol{\nu}))
\]
Hence, comparing it
with (\ref{refomulation}), we obtain $d_{\gamma\xi}(v)=d_{\gamma w_{b},\xi
}^{\prime}(v)\in\mathbb{N}%
[v,v^{-1}].$ We have established the desired positivity result:

\begin{theorem}
The polynomials $d_{{\boldsymbol{\lambda},\boldsymbol{\nu}}}(v)$ which appear
in (\ref{main_dec2}) have nonnegative integer coefficients.
\end{theorem}

\subsection{The case $v=1$}

The proof of the positivity we have obtained does not properly yield a
geometric interpretation of the coefficients $d_{\gamma\xi}(v).$ The purpose
of this section is to show that their specializations $d_{\gamma\xi}(1)$ may
be interpreted as composition multiplicities. Let us rewrite the right action in a
more coordinate free manner. For this, we consider
\[
\mathfrak{g}^{\prime}=[\mathfrak{g},\mathfrak{g}]=\mathfrak{sl}_{r}%
(\mathbb{C})\otimes\mathbb{C}[t,t^{-1}]\oplus\mathbb{C}c,
\]
where $\mathfrak{g}=\mathfrak{sl}_{r}(\mathbb{C})\otimes\mathbb{C}%
[t,t^{-1}]\oplus\mathbb{C}c\oplus\mathbb{C}d$ is the Kac-Moody Lie algebra of
type $A_{r-1}^{(1)}.$ Then the fundamental weights $\Lambda_{0},\dots
,\Lambda_{r-1}$ remain linearly independent on
\[
\mathfrak{h}^{\prime}=\bigoplus_{i=0}^{r-1}\mathbb{C}\alpha_{i}^{\vee}%
\]
and we may write its dual space as follows.
\[
{\mathfrak{h}^{\prime}}^{\ast}=\mathfrak{h}^{\ast}/\mathbb{C}\delta
=\bigoplus_{i=0}^{r-1}\mathbb{C}\Lambda_{i}.
\]
We identify the weight lattice $P$ of $\mathfrak{sl}_{r}(\mathbb{C})$ with the
set of level zero integral weights in ${\mathfrak{h}^{\prime}}^{\ast}$ by
\[
P=\frac{\bigoplus_{i=1}^{r}\mathbb{Z}\epsilon_{i}}{\mathbb{Z}(\epsilon
_{1}+\cdots+\epsilon_{r})}=\bigoplus_{i=1}^{r-1}\mathbb{Z}(\Lambda_{i}%
-\Lambda_{0})\subseteq{\mathfrak{h}^{\prime}}^{\ast}%
\]
where $\xi=\sum_{i=1}^{r}\xi_{i}\epsilon_{i}\mapsto\sum_{i=1}^{r-1}(\xi
_{i}-\xi_{i+1})(\Lambda_{i}-\Lambda_{0})$. \footnote{We drop ``modulo
$\mathbb{Z}(\epsilon_{1}+\cdots+\epsilon_{r})$ ''by abuse of notation.} For
$\xi\in P$, we define
\[
\hat{\xi}=-\xi+e\Lambda_{0}\in{\mathfrak{h}^{\prime}}^{\ast}.
\]
The Weyl group action on ${\mathfrak{h}^{\prime}}^{\ast}$ preserves
$P+e\Lambda_{0}$. Moreover, if we define $w\xi$ for $w\in W$ and $\xi\in P$ by
$w\hat{\xi}=-w\xi+e\Lambda_{0}$ where $\hat{\xi}\mapsto w\hat{\xi}$ is the Weyl group
action on ${\mathfrak{h}^{\prime}}^{\ast}$, then%
\[
s_{i}\xi=\xi_{i+1}\epsilon_{i}+\xi_{i}\epsilon_{i+1}+\sum_{j\neq i,i+1}\xi
_{j}\epsilon_{j},
\]
for $1\leq i\leq r-1$, and
\[
s_{0}\xi=(\xi_{r}-e)\epsilon_{1}+(\xi_{1}+e)\epsilon_{r}+\sum_{j\neq1,r}%
\xi_{j}\epsilon_{j}.
\]
Thus, $\xi\cdot w:=w^{-1}\xi$, for $\xi\in P$ and $w\in W$, is nothing but the right
action of $W$.

Let $J\subset\{1,\ldots,r-1\}$ and $\mu$ the composition of $r$ defined by
$J.$ Write $\mathfrak{p}_{\mu}(\mathbb{C})$ for the parabolic subalgebra of
$\mathfrak{g}$ defined by $\mu$ and $\mathfrak{l}_{\mu}(\mathbb{C})$ for the
standard Levi subalgebra of $\mathfrak{p}_{\mu}(\mathbb{C})$. For $\eta\in
P_{b}^{++}$, we denote by $V(w_{b}\hat{\eta})$ the finite dimensional
irreducible $\mathfrak{l}_{\mu}(\mathbb{C})\oplus\mathbb{C}c$-module with
highest weight $w_{b}\hat{\eta}-\rho$, where $\rho$ is such that $\langle
\rho,\alpha_{i}^{\vee}\rangle=1$, for $0\leq i\leq r-1$. Thus, the canonical
central element $c$ acts as the scalar $e-r$. We view $V(w_{b}\hat{\eta})$ as
a $\mathfrak p_{\mu}(\mathbb{C})\oplus\mathbb{C}c$-module. Then, through the
evaluation homomorphism
\[
\mathfrak p_{\mu}=\{X\in\mathfrak{sl}_{r}(\mathbb{C}[t])\mid X|_{t=0}%
\in\mathfrak p_{\mu}(\mathbb{C})\}\oplus\mathbb{C}c\rightarrow\mathfrak
p_{\mu}(\mathbb{C})\oplus\mathbb{C}c
\]
we may view it as a $\mathfrak p_{\mu}$-module as well. We define the
following $\mathfrak g^{\prime}$-module.
\[
M_{\mu}(w_{b}\hat{\eta})=U(\mathfrak g^{\prime})\otimes_{U(\mathfrak p_{\mu
})}V(w_{b}\hat{\eta}).
\]
If $X\in\mathfrak p_{\mu}$, then
\[
Xu\otimes v=[X,u]\otimes v+u\otimes Xv\quad(u\in U(\mathfrak g^{\prime}),v\in
V(w_{b}\hat{\eta})).
\]
Hence $M_{\mu}(w_{b}\hat{\eta})$ is isomorphic to the tensor product
representation of the adjoint representation on $U(\mathfrak g^{\prime})$ and
$V(w_{b}\hat{\eta})$ as a $\mathfrak p_{\mu}$-module. Thus $M_{\mu}(w_{b}%
\hat{\eta})$ is an integrable $\mathfrak p_{\mu}$-module.

For any $\zeta$ in ${\mathfrak{h}^{\prime}}^{\ast}$, we denote by $M(\zeta)$ the
Verma $\mathfrak g^{\prime}$-module with highest weight $\zeta-\rho$. Then, by
the Weyl character formula, we have for $\eta\in P_{b}^{++}$%
\[
M_{\mu}(w_{b}\hat{\eta})=\sum_{u\in W_{b}}(-1)^{\ell(u)}M(uw_{b}\hat{\eta
}).
\]

We consider the following maximal parabolic subalgebra of $\mathfrak
g^{\prime}$.
\[
\mathfrak g_{0}^{\prime}=\mathfrak{sl}_{r}(\mathbb{C}[t])\oplus\mathbb{C}%
c\subseteq\mathfrak g^{\prime}.
\]
We define
\[
M_{0}(w_{b}\hat{\eta})=U(\mathfrak g^{\prime})\otimes_{U(\mathfrak
g_{0}^{\prime})}L(w_{b}\hat{\eta})
\]
where $L(w_{b}\hat{\eta})$ is the irreducible highest weight $\mathfrak
g_{0}^{\prime}$-module whose highest weight is $w_{b}\hat{\eta}-\rho$.

Now, with the notation of \S\ \ref{sub-sec_stab}, observe that $\langle
\tilde{a},\alpha_{i}^{\vee}\rangle\leq0$, for $1\leq i\leq r-1$. Moreover, we
have
\[
-uw_{b}\eta=-uw_{b}\tilde{x}(\eta)^{-1}\tilde{a}
\]
such that $w_{\tilde{a}}\tilde{x}(\eta)w_{b}u^{-1}$ is the maximal length
coset representative of $W_{\tilde{a}}\tilde{x}(\eta)w_{b}u^{-1}$. Now we
apply the classical Kazhdan-Lusztig conjecture for semisimple Lie algebras,
which is the theorem by Beilinson-Bernstein and Brylinski-Kashiwara. Here, 
the Lie algebra is $\mathfrak{sl}_{r}(\mathbb{C})$ and it gives%
\[
M(uw_{b}\hat{\eta})=\sum_{\gamma\in\tilde{a}S_{r}}P_{w_{\tilde{a}}\tilde
{x}(\eta)w_{b}u^{-1},w_{\tilde{a}}\tilde{x}(\gamma)w_{b}}(1)M_{0}(w_{b}%
\hat{\gamma}),
\]
for $u\in W_{b}$. This implies that
\begin{align*}
M_{\mu}(w_{b}\hat{\eta})  &  =\sum_{u\in W_{b}}(-1)^{\ell(u)}M(uw_{b}%
\hat{\eta})\\
&  =\sum_{\gamma\in\tilde{a}S_{r}}P_{w_{\tilde{a}}\tilde{x}(\eta
)w_{b},w_{\tilde{a}}\tilde{x}(\gamma)w_{b}}^{J,-1}(1)M_{0}(w_{b}\hat{\gamma}).
\end{align*}
By the integrality as a $\mathfrak p_{\mu}$-module, we have
\[
M_{\mu}(w_{b}\hat{\eta})=\sum_{\gamma\in\tilde{a}S_{r}\cap P_{b}^{++}%
}P_{w_{\tilde{a}}\tilde{x}(\eta)w_{b},w_{\tilde{a}}\tilde{x}(\gamma)w_{b}%
}^{J,-1}(1)M_{0}(w_{b}\hat{\gamma}).
\]
Note also that $\hat{a}=-\sum_{i=1}^{r-1}(a_{i}-a_{i+1})(\Lambda_{i}%
-\Lambda_{0})+e\Lambda_{0}$ satisfies
\[
\langle\hat{a},\alpha_{i}^{\vee}\rangle=\left\{
\begin{array}
[c]{lll}%
a_{i+1}-a_{i}\geq0\quad & (1\leq i\leq r-1) & \\
e+a_{1}-a_{r}\geq1>0\quad & (i=0) &
\end{array}
\right.
\]
and we have
\[
uw_{b}\hat{\eta}=uw_{b}x(\eta)^{-1}\hat{a}
\]
such that $w_{a}x(\eta)w_{b}u^{-1}$ is the maximal length coset
representative of $_{a}\!Wx(\eta)w_{b}u^{-1}$, for $u\in W_{b}$. Thus, by the
Kazhdan-Lusztig conjecture again, this time for $\mathfrak g$,
\[
M(uw_{b}\hat{\eta})=\sum_{\xi\in aW}P_{w_{a}x(\eta)w_{b}u^{-1},w_{a}%
x(\xi)w_{b}}(1)L(w_{b}\hat{\xi}),
\]
for $u\in W_{b}$. This implies that
\begin{align*}
M_{\mu}(w_{b}\hat{\eta})  &  =\sum_{u\in W_{b}}(-1)^{\ell(u)}M(uw_{b}%
\hat{\eta})\\
&  =\sum_{\xi\in aW}P_{w_{a}x(\eta)w_{b},w_{a}x(\xi)w_{b}}^{J,-1}%
(1)L(w_{b}\hat{\xi}).
\end{align*}
By the integrality as a $\mathfrak p_{\mu}$-module again, we obtain
\[
M_{\mu}(w_{b}\hat{\eta})=\sum_{\xi\in aW\cap P_{b}^{++}}P_{w_{a}x(\eta
)w_{b},w_{a}x(\xi)w_{b}}^{J,-1}(1)L(w_{b}\hat{\xi}).
\]
Therefore, if we write
\[
M_{0}(w_{b}\hat{\gamma})=\sum_{\xi\in aW\cap P_{b}^{++}}d_{\gamma\xi}%
L(w_{b}\hat{\xi}),
\]
for $d_{\gamma\xi}\in\mathbb{N}$, in other words $[M_{0}(w_{b}\hat{\gamma
}):L(w_{b}\hat{\xi})]=d_{\gamma\xi}$, we have
\[
P_{w_{a}x(\eta)w_{b},w_{a}x(\xi)w_{b}}^{J,-1}(1)=\sum_{\gamma\in\tilde
{a}S_{r}\cap P_{b}^{++}}P_{w_{\tilde{a}}\tilde{x}(\eta)w_{b},w_{\tilde{a}%
}\tilde{x}(\gamma)w_{b}}^{J,-1}(1)d_{\gamma\xi}.
\]

Hence, we have the following interpretation of $d_{\boldsymbol{\lambda},\boldsymbol{\nu}}(1)$. 

\begin{proposition}
For the relative decomposition numbers evaluated at $v=1$, we have the equalities
\[
d_{\boldsymbol{\lambda},\boldsymbol{\nu}}(1)=[M_{0}(w_{b}
\hat{\gamma}):L(w_{b}\hat{\xi})]
\]
where $\xi=\zeta(\boldsymbol{\lambda})$ and $\gamma=\zeta(\boldsymbol{\nu})$.
\end{proposition}

It would be desirable to understand
$d_{\boldsymbol{\lambda},\boldsymbol{\nu}}(v)$ in terms of Jantzen filtration. In the case when $W_{b}$ is
trivial, we expect that the Verma module is rigid and Jantzen conjecture holds.\newline 

%
%
%
%
%

\noindent\textbf{Acknowledgements.} The first author is partly supported by
the Grant-in-Aid for Scientific Research (B) (No. 20340004), Japan Society for
the Promotion of Science. The second author is supported by "Agence nationale
de la recherche" ANR JC-07-1923-39. The third author is supported by "Agence
Nationale de la Recherche "ANR-09-JCJC-0102-01.

\end{document}